\documentclass[11pt,letterpaper]{amsart}
\usepackage{amsmath, amssymb, amsfonts, verbatim, graphicx, amsthm, tikz-cd, mathrsfs, color, comment, fancyhdr, amsrefs}
\usepackage{enumerate}
\usepackage{url}

\usepackage[foot]{amsaddr}

\allowdisplaybreaks[1]

\usepackage[sc]{mathpazo}
\usepackage[T1]{fontenc}

\newcommand{\vertiii}[1]{{\left\vert\kern-0.25ex\left\vert\kern-0.25ex\left\vert #1
		\right\vert\kern-0.25ex\right\vert\kern-0.25ex\right\vert}}

\theoremstyle{plain}
\newtheorem{Thm}{Theorem}[section]
\newtheorem{Prop}[Thm]{Proposition}
\newtheorem{Lem}[Thm]{Lemma}
\newtheorem{Cor}[Thm]{Corollary}

\newtheorem*{JK}{Jewett-Krieger Theorem}

\theoremstyle{definition}
\newtheorem{Def}[Thm]{Definition}
\newtheorem{Not}[Thm]{Notation}

\theoremstyle{remark}
\newtheorem{Rmk}[Thm]{Remark}
\newtheorem{Ex}[Thm]{Example}

\renewcommand{\epsilon}{\varepsilon}

\usepackage{setspace}

\usepackage{hyperref}

\title{Spatial-Temporal Differentiation Theorems}

\author{I. Assani$^1$, A. Young$^2$}
\address{University of North Carolina at Chapel Hill}

\email{$^1$assani@email.unc.edu, $^2$aidanjy@live.unc.edu}
\begin{document}
	
	\maketitle
	
	\begin{abstract}
		Let $(X, \mathcal{B}, \mu, T)$ be a dynamical system where $X$ is a compact metric space with Borel $\sigma$-algebra $\mathcal{B}$, and $\mu$ is a probability measure that is ergodic with respect to the homeomorphism $T : X \to X$. We study the following differentiation problem: Given $f \in C(X)$ and $F_k \in \mathcal{B}$, where $\mu(F_k) > 0$ and $\mu(F_k) \to 0$, when can we say that
		$$\lim_{k \to \infty} \frac{\int_{F_k} \left( \frac{1}{k} \sum_{i = 0}^{k - 1} T^i f \right) \mathrm{d} \mu}{\mu(F_k)} = \int f \mathrm{d} \mu?$$
	\end{abstract}
	
	Let $(X, \mathcal{B}, \mu, T)$ be an ergodic topological dynamical system, where $X$ is a compact metric space, $\mathcal{B}$ is  the Borel $\sigma$-algebra of $X$, and $T : X \to X$ a homeomorphism that is ergodic with respect to the probability measure $\mu$. We consider \emph{spatial-temporal differentiation problems} of the type
	$$\lim_{k \to \infty} \frac{1}{\mu(F_k)} \int_{F_k} \left( \frac{1}{k} \sum_{i = 0}^{k - 1} T^i f \right) \mathrm{d} \mu ,$$
	where $(F_k)_{k = 1}^\infty$ is a sequence of measurable sets $F_k \in \mathcal{B}$ with positive measure $\mu(F_k) > 0$; specifically, we consider questions of when this limit exists, and when it exists, what that limit is for $f \in L^\infty(X, \mu)$.
	
	Before proceeding, we pause to distinguish these problems from two other kinds of differentiation problems which we will call temporal and spatial differentiation problems. A temporal differentiation problem might look like $\lim_{k \to \infty} \frac{1}{k} \sum_{i = 0}^{k - 1} \int T^i f \mathrm{d} \mu$, and a spatial differentiation problem might look like $\lim_{k \to \infty} \frac{1}{\mu(F_k)} \int_{F_k} f \mathrm{d} \mu$, where $F_k$ are sets of positive measure. A temporal differentiation problem, at least one of this form, would be trivial so long as $T$ is measure-preserving, and results exist regarding spatial differentiation problems (e.g. the Lebesgue Differentiation Theorem \cite[3.21]{FollandReal}, Fundamental Theorem of Calculus). Our problem, however, fits in neither of these bins, except in trivial cases, and these differentiation problems might be called "spatial-temporal" differentiation problems.
	
	A spatial-temporal differentiation problem \linebreak$\left( \frac{1}{\mu(F_k)} \int_{F_k} \left( \frac{1}{k} \sum_{i = 0}^{k - 1} T^i f \right) \mathrm{d} \mu \right)_{k = 1}^\infty$ hinges on three parameters: the dynamical system $(X, \mathcal{B}, \mu, T)$, the sequence $(F_k)_{k = 1}^\infty$ of measurable sets, and the function $f \in L^\infty(X, \mu)$. For the most part, the questions we consider in this article can be understood as "fixing" two of these parameters and investigating what can be said about the convergence properties of the differentiation when the remaining parameter is allowed to "vary".
	
	The paper is organized as follows:
	\begin{enumerate}
		\item In Section \ref{Topological STDs}, we consider certain functions which behave particularly well with respect to these differentiations, called uniform functions, and analyze them with respect to these spatial-temporal differentiations. We pay special attention to topological dynamical systems and how these spatial-temporal differentiations interact with unique ergodicity and uniformity.
		\item In Section \ref{Non-expansive}, we consider a non-expansive  topological dynamical system, and consider spatial-temporal differentiations along certain random nested sequences of subsets, deriving probabilistic results.
		\item In Section \ref{Lipschitz}, we consider instead a broader class of Lipschitz maps, and differentiate along randomly chosen sequences of sets; in particular, we derive probabilistic results about spatial-temporal differentiations along random sequences of cylinders in a subshift, as well as find certain pathological counterexamples.
		\item In Section \ref{Bernoulli shifts and probability}, we turn to study differentiations along random cylinders on Bernoulli shifts, but using a more probabilistic set of tools different from those we employed in the second section. We then use these techniques to consider a different problem of random cylinders, where we allow the cylinders at different steps to have different centers.
	\end{enumerate}

	We thank the referee for their detailed and through comments.

	\section{Uniform functions and differentiation theorems}\label{Topological STDs}
	
	In this section, we consider questions of the following forms: Given an appropriate system $(X, \mathcal{B}, T, \mu)$, are there $f \in L^\infty(X, \mu)$ for which \linebreak
	$\left( \frac{1}{\mu(F_k)} \int_{F_k} \left( \frac{1}{k} \sum_{i = 0}^{k - 1} T^i f \right) \mathrm{d} \mu \right)_{k = 1}^\infty$ converges for \emph{all} choices of $(F_k)_{k = 1}^\infty$? On the other hand, are there restrictions we can place on $(X, \mathcal{B}, \mu, T)$ to ensure that $\left( \frac{1}{\mu(F_k)} \int_{F_k} \left( \frac{1}{k} \sum_{i = 0}^{k - 1} T^i f \right) \mathrm{d} \mu \right)_{k = 1}^\infty$ converges for \emph{all} choices of $(F_k)_{k = 1}^\infty$ \emph{and} all $f \in C(X)$? The answer to the former question will be centered around the notion of a uniform function (defined below), and the answer to the latter question will be centered around unique ergodicity.
	
	Let $X$ be a compact metrizable space with Borel $\sigma$-algebra $\mathcal{B}$, and let $T : X \to X$ be a homeomorphism. Then $(X, T)$ is uniquely ergodic iff the sequence $ \left( \frac{1}{k} \sum_{i = 0}^{k - 1} T^i f \right)_{k = 1}^\infty$ converges in $C(X)$ to a constant function for all $f \in C(X)$ \cite[Theorem 10.6]{EisnerOperators}, and when this  happens, the sequence converges to $\int f \mathrm{d} \mu$, where $\mu$ is the unique ergodic $T$-invariant Borel probability measure. Thus if $(F_k)_{k = 1}^\infty$ is \emph{any} sequence of measurable sets of positive measure, then $\frac{1}{k} \sum_{i = 0}^{k - 1} \alpha_{F_k} \left( T^i f \right) \to \int f \mathrm{d} \mu$ for all $f \in C(X)$, since $\alpha_{F_k}$ is a bounded functional on $L^{\infty}(X, \mu)$. Fix $\epsilon > 0$, and choose $K \in \mathbb{N}$ such that
	$$k \geq K \Rightarrow \left\| \int f \mathrm{d} \mu - \frac{1}{k} \sum_{i = 0}^{k - 1} T^i f \right\|_\infty \leq \epsilon.$$
	Then if $k \geq K$, we have
	\begin{align*}
		\left| \int f \mathrm{d} \mu - \frac{1}{k} \sum_{i = 0}^{k - 1} \alpha_{F_k} \left( T^i f \right) \right|	& = \left| \alpha_{F_k} \left( \int f \mathrm{d} \mu - \frac{1}{k} \sum_{i = 0}^{k - 1} T^i f \right) \right| \\
		& \leq \left\| \int f \mathrm{d} \mu - \frac{1}{k} \sum_{i = 0}^{k - 1} T^i f \right\|_\infty \\
		& \leq \epsilon .
	\end{align*}
	
	More generally, if we have any dynamical system $(Y, \mathcal{A}, \nu, S)$, we can call a function $g \in L^\infty(Y, \nu)$ \emph{uniform} if $\frac{1}{k} \sum_{i = 0}^{k - 1} S^i g \to \int g \mathrm{d} \nu$ in $L^\infty$. Let $\mathscr{U}(Y, \mathcal{A}, \nu, S) \subseteq L^\infty (Y, \nu)$ denote the space of all uniform functions on $(Y, \mathcal{A}, \nu, S)$. If $g$ is uniform, then for any sequence $(G_k)_{k = 1}^\infty$ of measurable sets of positive measure, we have
	$$\frac{1}{k} \sum_{i = 0}^{k - 1} \frac{1}{\nu(G_k)} \int_{G_k} S^i g \mathrm{d} \nu \to \int g \mathrm{d} \nu ,$$
	meaning that essentially any differentiation problem of the type that interests us will behave exceptionally well for that $g$.
	
	Whenever $(X, T)$ is a uniquely ergodic system, we have \linebreak $C(X) \subseteq \mathscr{U}(X, \mathcal{B}, \mu, T)$, since
	$$\left\| \int f \mathrm{d} \mu - \frac{1}{k} \sum_{i = 0}^{k - 1} T^i f \right\|_\infty \leq \left\| \int f \mathrm{d} \mu - \frac{1}{k} \sum_{i = 0}^{k - 1} T^i f \right\|_{C(X)} .$$
	
	We collect here a few results about some more general differentiation problems. We first demonstrate a general characterization theorem for uniform functions.
	
	\begin{Thm}\label{Uniformity and STDs}
		Let $(Y, \mathcal{A}, \nu, S)$ be an ergodic dynamical system, and let $g \in L^\infty (Y, \nu)$. Then $g$ is uniform if and only if for all sequences $(G_k)_{k = 1}^\infty$ in $\mathcal{A}$ of measurable sets of positive measure,
		$$\frac{1}{k} \sum_{i = 0}^{k - 1} \frac{1}{\nu(G_k)} \int_{G_k} S^i g \mathrm{d} \nu \to \int g \mathrm{d} \nu .$$
	\end{Thm}
	
	\begin{proof}
		$(\Rightarrow)$ If $g$ is uniform, then
		\begin{align*}
			\left| \int g \mathrm{d} \nu - \frac{1}{k} \sum_{i = 0}^{k - 1} \frac{1}{\nu(G_k)} \int_{G_k} S^i g \mathrm{d} \nu \right| & = \left| \frac{1}{\nu(G_k)} \int_{G_k} \left( \int g \mathrm{d} \nu - \frac{1}{k} \sum_{i = 0}^{k - 1} S^i g \right) \mathrm{d} \nu \right| \\
			& \leq \frac{1}{\nu(G_k)} \int_{G_k} \left\| \int g \mathrm{d} \nu - \frac{1}{k} \sum_{i = 0}^{k - 1} S^i g \right\|_\infty \\
			& = \left\| \int g \mathrm{d} \nu - \frac{1}{k} \sum_{i = 0}^{k - 1} S^i g \right\|_\infty \\
			& \to 0 .
		\end{align*}
		
		$(\Leftarrow)$ Suppose that $g$ is not uniform, and set $h_k = \int g \mathrm{d} \nu - \frac{1}{k} \sum_{i = 0}^{k - 1} S^i g $. Then $\limsup_{k \to \infty} \left\| h_k \right\|_\infty > 0$. Breaking $h_k$ into its real part $h_k^{\textrm{Re}}$ and imaginary part $h_k^{\textrm{Im}}$ tells us that either $\limsup_{k \to \infty} \left\| h_k^{\textrm{Re}} \right\|_\infty > 0$, or \linebreak$\limsup_{k \to \infty} \left\| h_k^{\textrm{Im}} \right\|_\infty > 0$. Suppose without loss of generality that \linebreak$\limsup_{k \to \infty} \left\| h_k^{\textrm{Re}} \right\|_\infty > 0$. Then at least one of the inequalities
		$$\nu \left( \left\{ y \in Y : h_{k}^{\textrm{Re}}(y) \geq \frac{\epsilon_0}{2} \right\} \right) > 0, \; \nu \left( \left\{ y \in Y : h_{k}^{\textrm{Re}}(y) \leq - \frac{\epsilon_0}{2} \right\} \right) > 0$$
		attains for infinitely many $k \in \mathbb{N}$. Assume without loss of generality that $I = \left\{ k \in \mathbb{N} : \nu \left( \left\{ y \in Y : h_{k}^{\textrm{Re}}(y) \geq \frac{\epsilon_0}{2} \right\} \right) > 0 \right\}$ is an infinite set.
		
		Construct a sequence $\left( G_k \right)_{k = 1}^\infty$ by letting $G_k = \left\{ y \in Y : h_{k}^{\textrm{Re}}(y) \geq \frac{\epsilon_0}{2} \right\}$ for all $k \in I$, and $G_k = Y$ for $k \in \mathbb{N} \setminus I$. Then if $k \in I$, then
		\begin{align*}
			\left| \frac{1}{\nu(G_k)} \int_{G_k} h_k \mathrm{d} \nu \right|	& \geq \left| \frac{1}{\nu(G_k)} \int_{G_k} h_k^{\textrm{Re}} \mathrm{d} \nu \right| \\
			& = \frac{1}{\nu(G_k)} \int_{G_k} h_k^{\textrm{Re}} \mathrm{d} \nu \\
			& \geq \frac{1}{\nu(G_k)} \int_{G_k} \frac{\epsilon_0}{2} \mathrm{d} \nu \\
			& = \frac{\epsilon_0}{2} .
		\end{align*}
		Therefore, there exist infinitely many $k \in \mathbb{N}$ such that
		$$\left| \int g \mathrm{d} \nu - \frac{1}{k} \sum_{i = 0}^{k - 1} \frac{1}{\nu(G_k)} \int_{G_k} S^i g \mathrm{d} \nu \right| = \left| \frac{1}{\nu(G_k)} \int_{G_k} h_k \mathrm{d} \nu \right| \geq \frac{\epsilon_0}{2},$$
		meaning that $\left| \int g \mathrm{d} \nu - \frac{1}{k} \sum_{i = 0}^{k - 1} \frac{1}{\nu(G_k)} \int_{G_k} S^i g \mathrm{d} \nu \right| \not \to 0$.
	\end{proof}

	Because we will so frequently be considering averages of functions over sets of positive measures, it will benefit us to introduce the following notation.
	
	\begin{Not}\label{Defining alpha}
	Let $(X, \mathcal{B}, \mu)$ be a probability space. When $F \in \mathcal{B}$ is a set of positive measure $\mu(F) > 0$, we denote by $\alpha_F$ the state on $L^\infty (X, \mu)$ given by
		\begin{align*}
			\alpha_F (f)	& : = \frac{1}{\mu(F)} \int_F f \mathrm{d} \mu .
		\end{align*}
	\end{Not}
	
	Theorem \ref{Uniformity and STDs} hints at why we consider spatial-temporal differentiations of $L^\infty$ functions instead of, for example, differentiations of $L^p$ functions for $p \in [1, \infty)$. One might plausibly propose that if we have a uniquely ergodic dynamical system $(X, \mathcal{B}, \mu, T)$, then we can observe that for all $f \in C(X)$, all spatial-temporal differentiations converge to $\int f \mathrm{d} \mu$. We could then try to extend this convergence to all of $L^1(X, \mu)$, since $C(X)$ is $L^1$-dense in $L^1(X, \mu)$. However, we know that a uniquely ergodic dynamical system can still have non-uniform $L^\infty$ functions (in fact, any ergodic dynamical system over a non-atomic standard probability space will have them, as seen in Proposition \ref{Existence of non-uniform functions}), so this cannot be right. The catch is that for measurable $F$ of nonzero measure, the functional $\alpha_F : f \mapsto \frac{1}{\mu(F)} \int_F f \mathrm{d} \mu$ is of norm $1$ with respect to $L^\infty$, but the same can't be said relative to $L^p$ for $p \in [1, \infty)$. As such, the "natural" choice of function for a spatial-temporal differentiation is an $L^\infty$ function.
	
	A similarly plausible but misguided attempt to establish convergence results of spatial-temporal differentiations for all $f \in L^\infty(X, \mu)$ could be through the concept of uniform sets. In \cite[Theorem 1]{HanselRaoult}, it was established that if $\mathcal{B}$ is separable with respect to the metric $(A, B) \mapsto \mu(A \Delta B)$, then there exists a dense $T$-invariant subalgebra $\mathcal{B}' \subseteq \mathcal{B}$ of sets such that $\chi_B$ is uniform for all $B \in \mathcal{B}'$. Again, one might propose that we could use a density argument to extend convergence results on spatial-temporal differentiations to functions $\chi_A$ for all $A \in \mathcal{B} \supseteq \mathcal{B}'$. But again, Theorem \ref{Uniformity and STDs} tells us that this would be tantamount to proving that all $L^\infty$ functions are uniform, and we know that there can exist non-uniform $L^\infty$ functions.
	
	Other results are possible regarding topological dynamical systems, as we show below.
	
	\begin{Lem}
		Let $f \in L^\infty(X, \mu)$ be a nonnegative function, where $(X, \mathcal{B}, \mu, T)$ is a dynamical system. Then the sequence $\left( \left\| \frac{1}{k} \sum_{i = 0}^{k - 1} S^i f \right\|_\infty \right)_{k = 1}^\infty$ is convergent, and
		$$\lim_{k \to \infty} \left\| \frac{1}{k} \sum_{i = 0}^{k - 1} S^i f \right\|_\infty = \inf_{k \in \mathbb{N}}  \left\| \frac{1}{k} \sum_{i = 0}^{k - 1} S^i f \right\|_\infty .$$
	\end{Lem}
	
	\begin{proof}
		Let $a_k = \left\| \sum_{i = 0}^{k - 1} T^i f \right\|_\infty$. Then the sequence $(a_k)_{k = 1}^\infty$ is subadditive. This follows since if $k, \ell \in \mathbb{N}$, then
		\begin{align*}
			a_{k + \ell}	& = \left\| \sum_{i =  0}^{k + \ell - 1} T^i f \right\|_\infty \\
			& \leq \left\| \sum_{i = 0}^{k - 1} T^i f \right\|_\infty + \left\| \sum_{i = k}^{k + \ell - 1} T^i g \right\|_\infty \\
			& = \left\| \sum_{i = 0}^{k - 1} T^i f \right\|_\infty + \left\| T^k \sum_{i = 0}^{\ell - 1} T^i f \right\|_\infty \\
			& = \left\| \sum_{i = 0}^{k - 1} T^i f \right\|_\infty + \left\| \sum_{i = 0}^{\ell - 1} T^i f \right\|_\infty \\
			& = a_k + a_\ell .
		\end{align*}
		The result then follows from the Subadditivity Lemma.
	\end{proof}
	
	\begin{Def}
		For nonnegative $f \in L^\infty(X, \mu)$, set
		$$\Gamma(f) : = \lim_{k \to \infty} \left\| \frac{1}{k} \sum_{i = 0}^{k - 1} T^i f \right\|_\infty .$$ We call this value $\Gamma(f)$ the \emph{gauge} of $f$.
	\end{Def}
	
	This $\Gamma(f)$ satisfies the inequality $\Gamma(f) \geq \int f \mathrm{d} \mu$, since
	\begin{align*}
		\frac{1}{k} \sum_{i = 0}^{k - 1} T^i f	& \leq \left\| \frac{1}{k} \sum_{i = 0}^{k - 1} T^i f \right\|_\infty \\
		\Rightarrow \int f \mathrm{d} \mu = \int \frac{1}{k} \sum_{i = 0}^{k - 1} T^i f \mathrm{d} \mu	& \leq \int \left\| \frac{1}{k} \sum_{i = 0}^{k - 1} T^i f \right\|_\infty \mathrm{d} \mu \\
		& = \left\| \frac{1}{k} \sum_{i = 0}^{k - 1} T^i f \right\|_\infty \\
		\Rightarrow \int f \mathrm{d} \mu	& \leq \inf_{k \in \mathbb{N}} \int \left\| \frac{1}{k} \sum_{i = 0}^{k - 1} T^i f \right\|_\infty \\
		& = \Gamma(f) .
	\end{align*}
	
	\begin{Def}
		Let $X$ be a compact metric space, and let $C_{\mathbb{R}}(X)$ denote the (real) space of real-valued continuous functions on $X$ endowed with the uniform norm $\| \cdot \|_{C(X)}$. Let $T : X  \to X$ be a continuous homeomorphism, and let $\mathcal{M}_T$ denote the family of all $T$-invariant Borel probability measures on $X$. A measure $\mu \in \mathcal{M}_T(X)$ is called \emph{$f$-maximizing} for some $f \in C_{\mathbb{R}}(X)$ if $\int f \mathrm{d} \mu = \sup_{\nu \in \mathcal{M}_T} \int f \mathrm{d} \nu$. We denote by $\mathcal{M}_{\mathrm{max}}(f)$ the space of all $f$-maximizing measures.
	\end{Def}
	
	The definition of maximizing measures is due to Jenkinson \cite[Definition 2.3]{Jenkinson}. The definition is topological in nature, in the sense that it is defined with reference to a homeomorphism on a compact metric space prior to any other measure that metric space might possess. A result of Jenkinson \cite[Proposition 2.4]{Jenkinson} tells us that for every $f \in C_{\mathbb{R}}(X)$, we have
	\begin{enumerate}
		\item $\mathcal{M}_{\mathrm{max}}(f) \neq \emptyset$,
		\item  $\mathcal{M}_{\mathrm{max}}(f)$ is a compact metrizable simplex, and
		\item the extreme points of  $\mathcal{M}_{\mathrm{max}}(f)$ are exactly the ergodic $f$-maximizing measures. In particular, every $f \in C_{\mathbb{R}}(X)$ admits an ergodic $f$-maximizing measure.
	\end{enumerate}
	
	For every $f \in C_{\mathbb{R}}(X)$, let $\mu_f$ denote an ergodic maximizing measure for $f$. We claim that $\Gamma(f) \leq \int f \mathrm{d} \mu_f$. To prove this, we note that \linebreak$\left\| \frac{1}{k} \sum_{i = 0}^{k - 1} T^i f \right\|_\infty \leq \max_{x \in X} \frac{1}{k} \sum_{i = 0}^{k - 1} T^i f (x)$, where the maximum exists because $X$ is compact and $f \in C_{\mathbb{R}}(X)$ is continuous. Choose $x_k \in X$ such that $\max_{x \in X} \frac{1}{k} \sum_{i = 0}^{k - 1} T^i f(x) = \frac{1}{k} \sum_{i = 0}^{k - 1} T^i f(x_k)$. Let $\delta_{x_k}$ denote the Borel point-mass probability measure
	$$\delta_{x_k}(A) = \begin{cases}
		1	& x_k \in A \\
		0	& x_k \not \in A
	\end{cases}$$
	Let $\mu_k = \frac{1}{k} \sum_{i = 0}^{k - 1} \delta_{T^i x_k}$, so that $\frac{1}{k} \sum_{i = 0}^{k - 1} T^i f (x_k) = \int f \mathrm{d} \mu_k$.
	
	Since the space of Borel probability measures on $X$ is compact in the weak* topology on $C(X)^*$, there exists a subsequence $(\mu_{k_n})_{n = 1}^\infty$ of $(\mu_k)_{k = 1}^\infty$ converging to a Borel probability measure $\mu'$. We claim that $\mu'$ is $T$-invariant, since if $g \in C(X)$, then
	\begin{align*}
		\left| \int T g \mathrm{d} \mu' - \int g \mathrm{d} \mu' \right|	& \leq \left| \int T g \mathrm{d} \mu' - \int  T g \mathrm{d} \mu_{x_{k_n}} \right| \\
		& + \left| \int T g \mathrm{d} \mu_{x_{k_n}} - \int g \mathrm{d} \mu_{x_{k_n}} \right| \\
		& + \left| \int g \mathrm{d} \mu_{x_{k_n}} - \int g \mathrm{d} \mu' \right| \\
		& \leq \left| \int T g \mathrm{d} \mu' - \int  T g \mathrm{d} \mu_{x_{k_n}} \right| \\
		& + \left| \frac{g \left( T^{k_n} x_{k_n} \right) - g(x_{k_n})}{k_n} \right| \\
		& + \left| \int g \mathrm{d} \mu_{x_{k_n}} - \int g \mathrm{d} \mu' \right| \\
		& \leq \left| \int T g \mathrm{d} \mu' - \int  T g \mathrm{d} \mu_{x_{k_n}} \right| \\
		& + \frac{2 \|g\|_{C(X)}}{k_n} \\
		& + \left| \int g \mathrm{d} \mu_{x_{k_n}} - \int g \mathrm{d} \mu' \right| \\
		& \to 0 .
	\end{align*}
	
	Therefore $\mu'$ is a $T$-invariant Borel probability measure on $X$ such that $\int f \mathrm{d} \mu' = \Gamma(f)$. But if $\mu_f$ is $f$-maximal, then
	$$\Gamma(f) = \int f \mathrm{d} \mu' \leq \int f \mathrm{d} \mu_f .$$ Under certain conditions, however, we can achieve equality here.
	
	\begin{Lem}\label{Assani Lemma 3}
		Let $(X, \mathcal{B}, \mu)$ be a probability space, where $X$ is a compact metric space with Borel $\sigma$-algebra on $X$ denoted by $\mathcal{B}$. Let $T : X \to X$ be a homeomorphism. If $\mu$ is strictly positive, and $f \in C_{\mathbb{R}}(X)$ is nonnegative, then
		$$\Gamma(f) = \int f \mathrm{d} \mu_f .$$
	\end{Lem}
	
	\begin{proof}
		First, we claim that if $X$ is compact and $\mu$ is strictly positive, then $\left\| \frac{1}{k} \sum_{i = 0}^{k - 1} T^i f \right\|_\infty = \sup_{x \in X} \left| \frac{1}{k} \sum_{i = 0}^{k - 1} T^i f(x) \right|$. Assume for contradiction that $\left\| \frac{1}{k} \sum_{i = 0}^{k - 1} T^i f \right\|_\infty < \sup_{x \in X} \left| \frac{1}{k} \sum_{i = 0}^{k - 1} T^i f(x) \right|$. Then there exists $R > 0$ such that $R < \sup_{x \in X} \left| \frac{1}{k} \sum_{i = 0}^{k - 1} T^i f(x) \right|$ and
		$$\mu \left( \left\{ x \in X : \left| \frac{1}{k} \sum_{i = 0}^{k - 1} T^i f(x) \right| > R \right\} \right) = 0 .$$
		But that set is nonempty and open, so it must have positive measure, a contradiction.
		
		Therefore $\left\| \frac{1}{k} \sum_{i = 0}^{k - 1} T^i f \right\|_\infty = \sup_{x \in X} \left| \frac{1}{k} \sum_{i = 0}^{k - 1} T^i f(x) \right|$ for all $k \in \mathbb{N}$. However, we can bound $\int f \mathrm{d} \mu_f$ by
		\begin{align*}
			\int f \mathrm{d} \mu_f	& = \int \frac{1}{k} \sum_{i = 0}^{k - 1} T^i f \mathrm{d} \mu \\
			& \leq \sup_{x \in X} \frac{1}{k} \sum_{i = 0}^{k - 1} T^i f(x) \mathrm{d} \mu \\
			& = \left\| \frac{1}{k} \sum_{i = 0}^{k - 1} T^i f \mathrm{d} \mu \right\|_\infty \\
			\int f \mathrm{d} \mu_f	& \leq \inf_{k \in \mathbb{N}} \left\| \frac{1}{k} \sum_{i = 0}^{k - 1} T^i f \mathrm{d} \mu \right\|_\infty \\
			& = \Gamma(f) ,
		\end{align*}
		establishing the opposite inequality.
	\end{proof}
	
	\begin{Lem}\label{Assani Lemma 4}
		Suppose that $(X, \mathcal{B}, \mu, T)$ consists of a compact metric space $X$ with Borel $\sigma$-algebra $\mathcal{B}$ and a strictly positive probability measure $\mu$ that is ergodic with respect to a homeomorphism $T: X \to X$. Then the system $(X, T)$ is uniquely ergodic if and only if $\Gamma(f) = \int f \mathrm{d} \mu$ for all nonnegative $f \in C_{\mathbb{R}}(X)$.
	\end{Lem}
	
	\begin{proof}
		$(\Rightarrow)$ If $(X, T)$ is uniquely ergodic, then in particular $\mu = \mu_f$ for all nonnegative $f \in C_{\mathbb{R}}(X)$. Therefore by the previous lemma, we have
		$$\int f \mathrm{d} \mu = \int f \mathrm{d} \mu_f = \Gamma(f) .$$
		
		$(\Leftarrow)$ If $(X, T)$ is \emph{not} uniquely ergodic, then we know that $\mathcal{M}_T(X)$ is not a singleton, and thus contains another ergodic measure $\nu$. By a result of Jenkinson \cite[Theorem 3.7]{Jenkinson}, we know that there exists $f \in C(X)$ real-valued such that $\nu = \mu_f$ is the \emph{unique} $f$-maximizing measure. We may assume without loss of generality that $f$ is nonnegative, since otherwise we can replace $f$ with $\tilde{f} - \inf_{x \in X} f(x)$. Since we claimed that $\nu$ was the unique $f$-maximizing measure, we can conclude in particular that
		\begin{align*}
			\int f \mathrm{d} \mu	& < \int f \mathrm{d} \nu \\
			& = \int f \mathrm{d} \mu_f \\
			& = \Gamma(f) .
		\end{align*}
	\end{proof}
	
	\begin{Thm}\label{Assani Theorem}
		Suppose that $(X, \mathcal{B}, \mu, T)$ consists of a compact metric space $X$ with Borel $\sigma$-algebra $\mathcal{B}$ and a probability measure $\mu$ that is ergodic with respect to a homeomorphism $T: X \to X$. Then the following results are related by the implications (1)$\Rightarrow$(2)$\Rightarrow$(3). Further, if $\mu$ is strictly positive, then (3)$\Rightarrow$(1).
		\begin{enumerate}
			\item $(X, T)$ is uniquely ergodic.
			\item For every sequence of Borel-measurable sets $(F_k)_{k = 1}^\infty$ of positive measure, and for every $f \in C(X)$, the limit $\lim_{k \to \infty} \alpha_{F_k} \left( \frac{1}{k} \sum_{i = 0}^{k - 1} T^i f \right)$ exists and is equal to $\int f \mathrm{d} \mu$, where $\alpha_\cdot$ is as defined in Notation \ref{Defining alpha}.
			\item For every sequence of open sets $(U_k)_{k = 1}^\infty$ of positive measure, and for every $f \in C(X)$, the limit $\lim_{k \to \infty} \alpha_{U_k} \left( \frac{1}{k} \sum_{i = 0}^{k - 1} T^i f \right)$ exists and is equal to $\int f \mathrm{d} \mu$, where $\alpha_\cdot$ is as defined in Notation \ref{Defining alpha}.
		\end{enumerate}
	\end{Thm}
	
	\begin{proof}
		(1)$\Rightarrow$(2): If $(X, T)$ is uniquely ergodic, then \linebreak $\left\| \int f \mathrm{d} \mu - \frac{1}{k} \sum_{i = 0}^{k - 1} T^i f \right\|_{C(X)} \stackrel{k \to \infty}{\to} 0$, so
		\begin{align*}
			\left| \int f \mathrm{d} \mu - \alpha_{F_k} \left( \frac{1}{k} \sum_{i = 0}^{k - 1} T^i f \right) \right|	& = \left| \alpha_{F_k} \left( \int f \mathrm{d} \mu - \frac{1}{k} \sum_{i = 0}^{k - 1} T^i f \right) \right| \\
			& \leq \left\| \int f \mathrm{d} \mu - \frac{1}{k} \sum_{i = 0}^{k - 1} T^i f \right\|_{C(X)} \\
			& \stackrel{k \to \infty}{\to} 0 .
		\end{align*}
		
		(2)$\Rightarrow$(3): Trivial, since an open set is automatically Borel.
		
		$\neg$(1)$\Rightarrow \neg$(3): Suppose $(X, T)$ is \emph{not} uniquely ergodic, and that $\mu$ is strictly positive. Then Lemma \ref{Assani Lemma 4} tells us that there exists nonnegative $f \in C_{\mathbb{R}}(X)$ for which $\Gamma(f) > \int f \mathrm{d} \mu$. Let $L$ be such that $\int f \mathrm{d} \mu < L < \Gamma(f)$, and consider the open set
		$$U_k = \left\{ x \in X : \frac{1}{k} \sum_{i = 0}^{k - 1} T^i f(x) > L \right\} .$$
		By the proof of Lemma \ref{Assani Lemma 3}, we know that
		$$\Gamma(f) = \inf_{k \in \mathbb{N}} \max_{x \in X} \frac{1}{k} \sum_{i = 0}^{k - 1} T^i f(x) = \inf_{k \in \mathbb{N}} \frac{1}{k} \sum_{i = 0}^{k - 1} T^i f(x_k) ,$$ where $x_k \in U_k$ for all $k \in \mathbb{N}$. Therefore $U_k$ is a nonempty open set, and since $\mu$ is strictly positive, that means $\mu(U_k) > 0$. Therefore
		\begin{align*}
			\alpha_{U_k} \left( \frac{1}{k} \sum_{i = 0}^{k - 1} T^i f \right)	& \geq \alpha_{U_k} \left( L \right) \\
			& = L \\
			\Rightarrow \liminf_{k \to \infty} \alpha_{U_k}(f)	& \geq L \\
			& > \int f \mathrm{d} \mu .
		\end{align*}
	\end{proof}
	
	\begin{Thm}\label{Pathological connected differentiation}
		Suppose that $(X, \mathcal{B}, \mu, T)$ consists of a compact connected metric space $X = (X, \rho)$ with Borel $\sigma$-algebra $\mathcal{B}$ and a probability measure $\mu$ that is ergodic with respect to a homeomorphism $T: X \to X$. Suppose further that $\mu$ is strictly positive, but $(X, T)$ is not uniquely ergodic. Then there exists a sequence $(U_k)_{k = 1}^\infty$ of nonempty open subsets of $X$ and a nonnegative continuous function $f \in C(X)$ such that the sequence $\left( \alpha_{U_k} \left( \frac{1}{k} \sum_{i = 0}^{k - 1} T^i f \right) \right)_{k = 1}^\infty$ is not Cauchy. Furthermore, if $\mu$ is atomless, then we can choose the sequence $(U_k)_{k = 1}^\infty$ such that $\mu \left( U_k \right) \searrow 0$.
	\end{Thm}
	
	\begin{proof}
		Lemma \ref{Assani Lemma 4} tells us that there exists nonnegative $f \in C_{\mathbb{R}}(X)$ for which $\Gamma(f) > \int f \mathrm{d} \mu$. Let $L, M \in \mathbb{R}$ such that $\int f \mathrm{d} \mu < L < M < \Gamma(f)$, and consider the open sets
		\begin{align*}
			V_k	& = \left\{ x \in X : \frac{1}{k} \sum_{i = 0}^{k - 1} T^i f(x) > M \right\} , \\
			W_k	& = \left\{ x \in X : \int f \mathrm{d} \mu < \frac{1}{k} \sum_{i = 0}^{k - 1} T^i f(x) < L \right\} .
		\end{align*}
		By the proof of Lemma \ref{Assani Lemma 3}, we know that $V_k \neq \emptyset$, so let $x_k \in V_k$. We also know that there exists $z_k$ in $X$ such that $f(z_k) \leq \int f \mathrm{d} \mu$, since if $f(z) > \int f \mathrm{d} \mu$ for all $z \in X$, then $\int f(z) \mathrm{d} \mu(z) > \int f \mathrm{d} \mu$, a contradiction. By the Intermediate Value Theorem, there then exists $y_k \in W_k$. Construct $(U_k)_{k = 1}^\infty$ as
		$$U_k = \begin{cases}
			V_k ,	& \textrm{$k$ odd} \\
			W_k ,	& \textrm{$k$ even} .
		\end{cases}$$
		Then
		\begin{align*}
			\limsup_{k \to \infty} \alpha_{U_{2k - 1}} \left( \frac{1}{2k - 1} \sum_{i = 0}^{2k - 2} T^i f \right)	& \geq \limsup_{k \to \infty} \alpha_{U_{2k - 1}} \left( M \right) \\
			& = M , \\
			\liminf_{k \to \infty} \alpha_{U_{2k}} \left( \frac{1}{2k} \sum_{i = 0}^{2k - 1} T^i f \right)	& \leq \liminf_{k \to \infty} \alpha_{U_{2k}} \left( L \right) \\
			& = L .
		\end{align*}
		Therefore
		$$\liminf_{k \to \infty} \alpha_{U_k} \left( \frac{1}{k} \sum_{i = 0}^{k - 1} T^i f \right) \leq L < M \leq \limsup_{k \to \infty} \alpha_{U_k} \left( \frac{1}{k} \sum_{i = 0}^{k - 1} T^i f \right) .$$
		
		Moreover, if $\mu$ is atomless, then we can choose $\left(U_k \right)_{k = 1}^\infty$ so that $\mu \left( U_k \right) \to 0$ by letting $U_k$ be a ball of sufficiently small radius contained in $V_k$ (if $k$ is odd) or $W_k$ (if $k$ is even). The above calculations can be carried out in the same way.
	\end{proof}
	
	In Theorem \ref{Pathological Differentiation}, we construct an example of a Bernoulli shift $(X, \mathcal{B}, \mu, T)$ where there exists $(x, f) \in X \times C(X)$ such that the sequence \linebreak $\left( \alpha_{C_k(x)} \left( \frac{1}{k} \sum_{i = 0}^{k - 1} T^i f \right) \right)_{k = 1}^\infty$ not only does not converge to $\int f \mathrm{d} \mu$, as in Theorem \ref{Assani Theorem}, but such that it does not converge at all. Theorem \ref{Pathological connected differentiation} does not encompass that example, since subshifts are a priori totally disconnected.
	
	In the next result, we will be making use of the Jewett-Krieger Theorem in a specific formulation. This is the formulation originally proven by Jewett in \cite{Jewett} under the assumption that the transformation was weakly mixing; Bellow and Furstenberg later demonstrated in \cite{BellowFurstenberg} that the parts of Jewett's argument which relied on the weakly mixing property could be proven under the weaker assumption of ergodicity. The version of the Jewett-Krieger Theorem we will be using is as follows.
	
	\begin{JK}
		Given an invertible ergodic system $(Y, \mathcal{A}, \nu, S)$ on a standard probability space $(Y, \mathcal{A}, \nu)$, there exists an essential isomorphism $h : (Y, \mathcal{A}, \nu,  S) \to \left(2^\omega, \mathcal{B}, \mu, T \right)$ (where $2^\omega$ denotes the Cantor space) such that $\left( 2^\omega, T \right)$ is a strictly ergodic system.
	\end{JK}
	
	The following result provides some structure statements about the space $\mathscr{U}(Y, \mathcal{A}, \nu, S)$ of uniform functions.
	
	\begin{Thm}
		Let $(Y, \mathcal{A}, \nu)$ be a standard probability space, and $S : Y \to Y$ an ergodic automorphism. Then $\mathscr{U}(Y, \mathcal{A}, \nu, S)$ is a closed $S$-invariant subspace of $L^\infty (Y, \nu)$ that is closed under complex conjugation, and contains a unital $S$-invariant C*-subalgebra $A$ which is dense in $L^1(Y, \nu)$. This $A$ is isomorphic as a C*-subalgebra to $C \left( 2^\omega \right)$.
	\end{Thm}
	
	\begin{proof}
		First, we prove that $\mathscr{U}(Y, \mathcal{A}, \nu, S)$ is a closed $S$-invariant subspace of $L^\infty(Y, \nu)$. The fact it is a subspace of $L^\infty(Y, \nu)$ is clear, so suppose $f \in \operatorname{cl} \left( \mathscr{U}(Y, \mathcal{A}, \nu, S) \right)$. Then there exists $g \in \mathscr{U}(Y, \mathcal{A}, \nu, S)$ such that $\left\|f - g \right\|_\infty \leq \frac{\epsilon}{3}$. Choose $K \in \mathbb{N}$ such that $k \geq K \Rightarrow \left\| \int g \mathrm{d} \mu - \frac{1}{k} \sum_{i = 0}^{k - 1} T^i g \right\|_\infty \leq \frac{\epsilon}{3}$ Then
		\begin{align*}
			\left\| \int f \mathrm{d} \mu - \frac{1}{k} \sum_{i = 0}^{k - 1} T^i f \right\|_\infty	& \leq \left\| \int f \mathrm{d} \mu - \int g \mathrm{d} \mu \right\|_\infty  \\
			& + \left\| \int g \mathrm{d} \mu - \frac{1}{k} \sum_{i = 0}^{k - 1} T^i g \right\|_\infty \\
			& + \left\| \frac{1}{k} \sum_{i = 0}^{k - 1} T^i (g - f) \right\|_\infty \\
			& \leq \left\| f - g \right\|_\infty + \left\| \int g \mathrm{d} \mu - \frac{1}{k} \sum_{i = 0}^{k - 1} T^i g \right\|_\infty + \left\| f - g \right\|_\infty \\
			& \leq \frac{\epsilon}{3} + \frac{\epsilon}{3} + \frac{\epsilon}{3} = \epsilon .
		\end{align*}
		Thus $f \in \mathscr{U}(Y, \mathcal{A}, \nu, S)$. Now, we claim that if $f \in \mathscr{U}(Y, \mathcal{A}, \nu, S)$, then $Sf, S^{-1}f \in \mathscr{U}(Y, \mathcal{A}, \nu, S)$. We compute
		\begin{align*}
			\left\| \int S f \mathrm{d} \nu - \frac{1}{k} \sum_{i = 0}^{k - 1} S^i (Sf) \right\|_\infty	& = \left\| \int f \mathrm{d} \nu - \frac{1}{k} \sum_{i = 0}^{k - 1} S^i (Sf) \right\|_\infty	\\
			& = \left\| \int f \mathrm{d} \nu - \left( \frac{1}{k} \sum_{i = 0}^{k - 1} S^i f \right) + \frac{1}{k} \left( f - S^k f \right) \right\|_\infty \\
			& \leq \left\| \int f \mathrm{d} \nu - \frac{1}{k} \sum_{i = 0}^{k - 1} S^i f \right\|_\infty + \frac{2 \| f \|}{k} \\
			& \stackrel{k \to \infty}{\to} 0.
		\end{align*}
		An analogous argument will show that $S^{-1} f \in \mathscr{U}(Y, \mathcal{A}, \nu, S)$. To see that $\mathscr{U}(Y, \mathcal{A}, \nu, S)$ is also closed under complex conjugation, we see that
		\begin{align*}
			\left\| \int \overline{f} \mathrm{d} \mu - \frac{1}{k} \sum_{i = 0}^{k - 1} T^i \overline{f} \right\|_\infty	& = \left\| \overline{\int f \mathrm{d} \mu - \frac{1}{k} \sum_{i = 0}^{k - 1} T^i f} \right\|_\infty \\
			& = \left\| \int f \mathrm{d} \mu - \frac{1}{k} \sum_{i = 0}^{k - 1} T^i f \right\|_\infty .
		\end{align*}
		
		Finally, we prove that $\mathscr{U}(Y, \mathcal{A}, \nu, S)$ contains a unital $S$-invariant C*-algebra $A$ that's dense in $L^1(Y, \nu)$. By the Jewett-Krieger Theorem, we know there exists an essential isomorphism $\phi : (Y, \mathcal{A}, \nu, S) \to \left(2^\omega, \mathcal{B}, \mu, T\right)$, where $\left( 2^\omega, \mathcal{B}, \mu, T \right)$ is uniquely ergodic. Let $A = \Phi \left( C \left( 2^\omega \right) \right)$, where $\Phi : L^\infty \left( 2^\omega , \mu \right) \to L^\infty(Y, \nu)$ is the pullback of $\phi$. Since $C \left( 2^\omega \right)$ is dense in $L^1 \left( 2^\omega , \mu \right)$, we can infer that $A = \Phi \left( C \left( 2^\omega \right) \right)$ is dense in $L^1(Y, \nu)$. Since continuous functions in a uniquely ergodic system are uniform, it follows that the functions of $A$ are uniform.
		
		Because $\mu$ is strictly positive, we know that $C \left( 2^\omega \right)$ is isomorphic to its copy in $L^\infty \left( 2^\omega , \mu \right)$ (see proof of Lemma \ref{Assani Lemma 3}), so this map $\Phi$ is an isomorphism between $C \left( 2^\omega \right) \subsetneq L^\infty \left( 2^\omega , \mu \right)$ and $A = \Phi \left( C \left( 2^\omega \right) \right)$.
	\end{proof}
	
	\begin{Prop}
		Suppose that $(X, \mathcal{B}, \mu, T)$ consists of a compact metric space $X = (X, \rho)$ with Borel $\sigma$-algebra $\mathcal{B}$ and a probability measure $\mu$ that is ergodic with respect to a homeomorphism $T: X \to X$, where $X$ is connected. Suppose further that $\exists F \in \mathcal{B}$ such that $0 < \mu(F) < 1$. Then there exists $f \in \mathscr{U}(X, \mathcal{B}, \mu, T) \setminus C(X)$.
	\end{Prop}
	
	\begin{proof}
		By the Jewett-Krieger Theorem, there exists an essential isomorphism $h : (X, \mathcal{B}, \mu, T) \to \left( X', \mathcal{B}', \mu', T' \right)$, where $X' = 2^\omega$ and $\left( X', T' \right)$ is uniquely ergodic. The topological space $2^\omega$ admits a basis $\mathcal{G}$ of clopen sets. We claim that there exists $G \in \mathcal{G}$ such that $0 < \mu'(G) < 1$.
		
		Assume for contradiction that $\mu'(G) \in \{0, 1\}$ for all $G \in \mathcal{G}$. If $(E_k)_{k = 1}^\infty$ is some sequence in $\mathcal{B}'$ of sets for which $\mu'(E_k) \in \{0, 1\}$, then
		\begin{align*}
			\mu' \left( \bigcup_{k = 1}^\infty E_k \right)	& = \max_{k \in \mathbb{N}}	\mu'(E_k)	& \in \{0, 1\}, \\
			\mu' \left( \bigcap_{k = 1}^\infty E_k \right)	& = \min_{k \in \mathbb{N}}	\mu' (E_k)	& \in \{0, 1\}, \\
			\mu'(X \setminus E_1)	& = 1 - \mu(E_1)	& \in \{0, 1\} .
		\end{align*}
		But since $\mathcal{G}$ generates $\mathcal{B}'$, this would imply that $\mu'(E) \in \{0, 1\}$ for all $E \in \mathcal{B}'$, a contradiction.
		
		Therefore, there exists $G_0 \in \mathcal{B}'$ clopen such that $0 < \mu'(G_0) < 1$. Set $g = \chi_{G_0} \in C \left( X' \right) \subseteq \mathscr{U} \left( X', \mathcal{B}', \mu', T' \right)$, and let $f = g \circ h$. Then $f \in \mathscr{U}(X, \mathcal{B}, \mu, T)$. But since $f$ takes values in $\{0, 1\}$, and $\mu(\{x \in X : f(x) = 1 \}) \not \in \{0, 1\}$, we must conclude that $f \in \mathscr{U} \left( X', \mathcal{B}', \mu', T' \right) \setminus C(X)$.
	\end{proof}
	
	We conclude this section by remarking that in most situations, we'll have $\mathscr{U}(Y, \mathcal{A}, \nu, S) \neq L^\infty(Y, \nu)$. We cite here a special case of a result of N. Ormes.
	
	\begin{Lem}
	Suppose $(Y, \mathcal{A}, \nu)$ is a non-atomic standard probability space, and $S : Y \to Y$ is an ergodic automorphism. Then there exists a minimal homeomorphism $T : 2^\omega \to 2^\omega$ and an affine homeomorphism $p : [0, 1] \to \mathcal{M}_T \left( 2^\omega \right)$ for which $\left( 2^\omega , \mathcal{B} , p(0) , T \right)$ is essentially isomorphic to $(Y, \mathcal{A}, \nu, S)$, where $\mathcal{B}$ here denotes the Borel $\sigma$-algebra on $2^\omega$.
	\end{Lem}

	\begin{proof}
	This is a special case of \cite[Corollary 7.4]{Ormes}, where we specifically consider the Choquet simplex $[0, 1]$.
	\end{proof}
	
	Since $\left( 2^\omega, T \right)$ is not uniquely ergodic, it follows that there exists $f_0 \in C \left( 2^\omega \right)$ such that $\left( \frac{1}{k} \sum_{i = 0}^{k - 1} T^i f_0 \right)_{k = 1}^\infty$ does not converge uniformly to the constant $\int f_0 \mathrm{d} \left( p(0) \right)$. Since $\left( 2^\omega , T \right)$ is minimal, and the support of $p(0)$ is a nonempty $T$-invariant compact subset of $2^\omega$, it follows that $p(0)$ is strictly positive, and so the uniform norm on $C \left( 2^\omega \right)$ coincides with the $L^\infty \left( 2^\omega , p(0) \right)$ norm on $C \left( 2^\omega \right)$. As such, it follows that
	$$\left\| \int f_0 \mathrm{d} (p(0)) - \frac{1}{k} \sum_{i = 0}^{k - 1} T^i f_0 \right\|_\infty = \sup_{x \in 2^\omega} \left| \int f_0 \mathrm{d} (p(0)) - \frac{1}{k} \sum_{i = 0}^{k - 1} T^i f_0(x) \right| \stackrel{k \to \infty}{\not \to} 0 .$$ Let $\phi : (Y, \mathcal{A}, \nu, S) \to \left( 2^\omega , \mathcal{B}, p(0), T \right)$ be an essential isomorphism, and let $\Phi : L^\infty \left( 2^\omega , p(0) \right) \to L^\infty (Y, \nu)$ be the pullback of $\phi$. Then
	$$\left\| \int \left( \Phi f_0 \right) \mathrm{d} \nu - \frac{1}{k} \sum_{i = 0}^{k - 1} S^i \left( \Phi f_0 \right) \right\|_\infty = \left\| \int f_0 \mathrm{d} (p(0)) - \frac{1}{k} \sum_{i = 0}^{k - 1} T^i f_0 \right\|_\infty \stackrel{k \to \infty}{\not \to} 0 .$$
	Therefore $\Phi f_0 \in L^\infty(Y, \nu) \setminus \mathscr{U}(Y, \mathcal{A}, \nu, S)$.
	
	The following proposition summarizes this discussion.
	
	\begin{Prop}\label{Existence of non-uniform functions}
		Suppose $(Y, \mathcal{A}, \nu)$ is a non-atomic standard probability space, and $S : Y \to Y$ is an ergodic automorphism. Then $\mathscr{U}(Y, \mathcal{A}, \nu, S) \neq L^\infty (Y, \nu)$.
	\end{Prop}
	
	\section{Non-expansive maps}\label{Non-expansive}
	In this section, as well as in Section \ref{Lipschitz}, we investigate for a certain class of dynamical system $(X, \mathcal{B}, \mu, T)$ what can be said about the convergence properties of $\left( \frac{1}{\mu(F_k)} \int_{F_k} \left( \frac{1}{k} \sum_{i = 0}^{k - 1} T^i f \right) \mathrm{d} \mu \right)_{k = 1}^\infty$ for $f \in C(X)$ when we consider a "probabilistically generic" sequence $(F_k)_{k = 1}^\infty$. In other words, we investigate in some sense a "typical" behavior of \linebreak$\left( \frac{1}{\mu(F_k)} \int_{F_k} \left( \frac{1}{k} \sum_{i = 0}^{k - 1} T^i f \right) \mathrm{d} \mu \right)_{k = 1}^\infty$ for $f \in C(X)$, and find sufficient conditions for this differentiation to converge almost surely to $\int f \mathrm{d} \mu$ for all $f \in C(X)$.
	
	Let $X = (X, \rho)$ be a compact metric space, and $T : X \to X$ a $1$-Lipschitz map, i.e. such that $\rho(Tx, Ty) \leq \rho(x, y)$ for all $x, y \in X$. Let $\mathcal{B}$ denote the Borel $\sigma$-algebra on $X$, and $\mu$ a $T$-invariant, ergodic Borel probability measure on $X$. Then $(X, T)$ has topological entropy $0$, and thus $(X, \mathcal{B}, \mu, T)$ is automatically of entropy $0 < \infty$ \cite[Lemma 1]{Goodman71}. By the Krieger Generator Theorem \cite[2.1]{KriegerGenerator}, the ergodic system admits a finite measurable partition $\mathcal{E} = \{ E_d \}_{d \in \mathcal{D}}$ of $X$ such that $\left\{ T^i E_d : i \in \mathbb{Z}, d \in \mathcal{D} \right\}$ generates the $\sigma$-algebra $\mathcal{B}$, where $\mathcal{D}$ is a finite indexing set. We call $\mathcal{E}$ a \emph{generator} of $(X, \mathcal{B}, \mu, T)$.
	
	Let $d_i : X \to \mathcal{D}, i \in \mathbb{Z}$ be the measurable random variable uniquely determined by the relation
	$$x \in T^{-i} E_{d_i(x)} ,$$
	or equivalently
	$$T^i x \in E_{d_i(x)}$$
	Given a word $\mathbf{a} = (a_0, a_1, \ldots, a_{\ell - 1}) \in \mathcal{D}^\ell$, we define the \emph{cylinder associated to $\mathbf{a}$} by
	$$[a_0, a_1, \ldots, a_{\ell - 1}] : = \bigcap_{i = 0}^{\ell - 1} T^{-i} E_{a_i} .$$
	We also define the \emph{rank-$k$ cylinder associated to $x \in X$} by
	$$C_k(x) : = [d_0(x), d_1(x), \ldots, d_{k - 1}(x)] = \bigcap_{i = 0}^{k - 1} T^{-i} E_{d_i(x)} .$$
	Equivalently, we can define $C_k(x)$ to be the element of $\lor_{i = 0}^{k - 1} T^{-i} \mathcal{E}$ containing $x$.
	
	We note here that $\mu(C_k(x)) > 0$ for all $k \in \mathbb{N}$ for almost all $x \in X$, since
	\begin{align*}
		\{ x \in X : \exists k \in \mathbb{N} \textrm{ s.t. } \mu(C_k(x)) = 0 \}	& = \bigcup_{k \in \mathbb{N}} \left\{ x \in X : \mu(C_k(x)) = 0 \right\} \\
		& = \bigcup_{k \in \mathbb{N}} \left( \bigcup_{ \mathbf{d} \in \mathcal{D}^k \textrm{ s.t. } \mu([\mathbf{d}]) = 0 } [\mathbf{d}] \right)
	\end{align*}
	is a countable union of null sets.
	
	Suppose further that $\operatorname{diam} (C_k(x)) \to 0$ for almost all $x \in X$.	Our main result for this section is the following.
	
	\begin{Thm}\label{Non-expansinve random cylinders}
		Let $X = (X, \rho)$ be a compact metric space, and $T : X \to X$ a $1$-Lipschitz map, i.e. such that $\rho(Tx, Ty) \leq \rho(x, y)$ for all $x, y \in X$. Let $\mathcal{B}$ denote the Borel $\sigma$-algebra on $X$, and $\mu$ a $T$-invariant, ergodic Borel probability measure on $X$. Let $\mathcal{E} = \{ E_d \}_{d \in \mathcal{D}}$ be a finite measurable partition of $X$ which generates $\mathcal{B}$, and let $C_k(x)$ be the element of $\lor_{i = 0}^{k - 1} T^{-i} \mathcal{E}$ containing $x$. Suppose further that $\operatorname{diam}(C_k(x)) \to 0$ for almost all $x \in X$. Then the set of $x \in X$ such that
		$$\frac{1}{\mu(C_k(x))} \int_{C_k(x)} \frac{1}{k} \sum_{i = 0}^{k - 1} T^i f \mathrm{d} \mu \stackrel{k \to \infty}{\to } \int f \mathrm{d} \mu$$
		for all $f \in C(X)$ is of full measure.
	\end{Thm}
	
	\begin{proof}
		Since $X$ is compact metrizable, we know that $C(X)$ is a separable vector space, so let $\{ f_n \}_{n \in \mathbb{N}}$ be a countable set in $C(X)$ such that $\overline{\operatorname{span}}\{ f_n \}_{n \in \mathbb{N}} = C(X)$, where the closure is taken in the uniform norm on $C(X)$. Let
		$$S_n = \left\{ x \in X : \frac{1}{k} \sum_{i = 0}^{k - 1} \alpha_{C_k(x)} \left( T^i f_n \right) \to \int f_n \mathrm{d} \mu \right\} .$$
		We claim that $\mu(S_n) = 1$.
		
		Let $x \in X$ such that $\operatorname{diam}(C_k(x)) \to 0$, that $\mu(C_k(x)) > 0$ for all $k \in \mathbb{N}$, and such that $\frac{1}{k} \sum_{i = 0}^{k - 1} T^i f_n (x) \to \int f_n \mathrm{d} \mu$. By the Birkhoff ergodic theorem \cite[Theorem 1.5]{Walters}, the set of all such $x$ is of full measure. Fix $\epsilon > 0$. Since $f_n$ is uniformly continuous, we know there exists $\delta > 0$ such that $\rho(x_1, x_2) \leq \delta \Rightarrow |f_n(x_1) - f_n(x_2)| \leq \frac{\epsilon}{2}$. Choose $K_1 \in \mathbb{N}$ such that $\operatorname{diam}(C_k (x)) \leq \delta$ for all $k \geq K_1$. Choose $K_2 \in \mathbb{N}$ such that $k \geq K_2 \Rightarrow \left| \int f_n \mathrm{d} \mu - \frac{1}{k} \sum_{i = 0}^{k - 1} T^i f_n (x) \right| \leq \frac{\epsilon}{2}$. Let $K = \max \{ K_1, K_2 \}$, and suppose that $k \geq K$. Then
		\begin{align*}
			& \left| \int f_n \mathrm{d} \mu - \frac{1}{k} \sum_{i = 0}^{k - 1} \alpha_{C_k(x)} \left(T^i f_n\right) \right|	\\
			& \leq \left| \int f_n \mathrm{d} \mu - \frac{1}{k} \sum_{i = 0}^{k - 1} T^i f_n(x) \right| + \left| \frac{1}{k} \sum_{i = 0}^{k - 1} T^i f_n (x) - \frac{1}{k} \sum_{i = 0}^{k - 1} \alpha_{C_k(x)} \left( T^i f_n \right) \right| \\
			& \leq \frac{\epsilon}{2} + \frac{1}{k} \sum_{i = 0}^{k - 1} \left| T^i f_n(x) - \frac{1}{\mu(C_k(x))} \int_{C_k(x)} T^i f \right| \\
			& = \frac{\epsilon}{2} + \frac{1}{k} \sum_{i = 0}^{k - 1} \left| \frac{1}{\mu(C_k(x))} \int_{C_k(x)} T^i f_n(x) - T^i f_n \mathrm{d} \mu \right| \\
			& = \frac{\epsilon}{2} + \frac{1}{k} \sum_{i = 0}^{k - 1} \left| \frac{1}{\mu \left( T^i C_k(x) \right)} \int_{T^i C_k(x)} f_n(x) - f_n \mathrm{d} \mu \right| \\
			& \leq \frac{\epsilon}{2} + \frac{1}{k} \sum_{i = 0}^{k - 1} \frac{1}{\mu \left( T^i C_k(x) \right)} \int_{T^i C_k(x)} \left| f_n(x) - f_n \right| \mathrm{d} \mu \\
			& \leq \frac{\epsilon}{2} + \frac{1}{k} \sum_{i = 0}^{k - 1} \frac{1}{\mu \left( T^i C_k(x) \right)} \int_{T^i C_k(x)} \frac{\epsilon}{2} \mathrm{d} \mu \\
			& =  \epsilon ,
		\end{align*}
		since $\operatorname{diam} \left( T^i C_k(x) \right) \leq \operatorname{diam} (C_k(x)) \leq \operatorname{diam} (C_K(x)) < \delta$. Thus if \linebreak$\mu(C_k(x))> 0$ for all $k \in \mathbb{N}$, if $\operatorname{diam}(C_k(x)) \to 0$, and if $\frac{1}{k} \sum_{i = 0}^{k - 1} T^i f_n (x) \to \int f_n \mathrm{d} \mu$, then $x \in E_n$. Thus $\mu(S_n) = 1$ for all $n \in \mathbb{N}$, and so $\mu \left( \bigcap_{n \in \mathbb{N}} S_n \right) = 1$.
		
		We claim now that if $x \in S = \bigcap_{n \in \mathbb{N}} S_n$, then $\frac{1}{k} \sum_{i = 0}^{k - 1} \alpha_{C_k(x)} \left( T^i f \right) \to \int f \mathrm{d} \mu$ for all $f \in C(X)$. Fix $x \in S, f \in C(X), \epsilon > 0$. Then there exist $N \in \mathbb{N}$ and $z_1, \ldots, z_N \in \mathbb{C}$ such that
		$$\left\| f - \sum_{n = 1}^{N} z_n f_n \right\|_\infty < \frac{\epsilon}{3} .$$
		Choose $L_1, \ldots, L_N \in \mathbb{N}$ such that
		$$k \geq L_n \Rightarrow \left| \int f_n \mathrm{d} \mu - \frac{1}{k} \sum_{i = 0}^{k - 1} \alpha_{C_k(x)} \left( T^i f_n \right) \right| < \frac{\epsilon}{3 N \max \{ |z_1|, \ldots, |z_N|, 1 \}} .$$
		Abbreviate $g = \sum_{n = 1}^N z_n f_n$, and let $L = \max \{ L_1, \ldots, L_N \}$. Then if $k \geq L$, then
		\begin{align*}
			& \left| \int f \mathrm{d} \mu - \frac{1}{k} \sum_{i = 0}^{k - 1} \alpha_{C_k(x)} \left( T^i f \right) \right|	\\
			& \leq \left| \int f \mathrm{d} \mu - \int g \mathrm{d} \mu \right| + \left| \int g \mathrm{d} \mu - \frac{1}{k} \sum_{i = 0}^{k - 1} \alpha_{C_k(x)} \left( T^i g \right) \right| \\
			& + \left| \frac{1}{k} \sum_{i = 0}^{k - 1} \alpha_{C_k(x)} \left( T^i (g - f) \right) \right| \\
			& \leq \left\| f - g \right\|_\infty + \sum_{n = 1}^N |z_n| \left| \int f_n \mathrm{d} \mu - \frac{1}{k} \sum_{i = 0}^{k - 1} \alpha_{C_k(x)} \left( T^i f_n \right) \right| \\
			& + \frac{1}{k} \sum_{i = 0}^{k - 1} \left\| g - f \right\|_\infty \\
			& \leq \epsilon .
		\end{align*}
		Thus $x \in S \Rightarrow \lim_{k \to \infty} \frac{1}{k} \sum_{i = 0}^{k - 1} \alpha_{C_k(x)} \left( T^i f \right) = \int f \mathrm{d} \mu$ for all $f \in C(X)$. Since $\mu(S) = 1$, this concludes the proof.
	\end{proof}
	
	\begin{Rmk}
	We remark that the cylindrical structure of the $C_k(x)$ was not essential to our proof of Theorem \ref{Non-expansinve random cylinders}. Rather, the important feature of $(C_k(x))_{k = 1}^\infty$ was that their diameter went to $0$ as $k \to \infty$. To demonstrate this fact, we consider the scenario where we replace the $C_k(x)$ with balls around $x$ of radius decreasing to $0$, and note that the technique of proof is remarkably similar to that used to prove Theorem \ref{Non-expansinve random cylinders}.
	\end{Rmk}
	
	\begin{Thm}\label{Non-expansive random balls}
		Let $X = (X, \rho)$ be a compact metric space, and $T : X \to X$ a $1$-Lipschitz map, i.e. such that $\rho(Tx, Ty) \leq \rho(x, y)$ for all $x, y \in X$. Let $\mathcal{B}$ denote the Borel $\sigma$-algebra on $X$, and $\mu$ a $T$-invariant, ergodic Borel probability measure on $X$. Let $(r_k)_{k = 1}^\infty$ be a non-increasing sequence of positive numbers $r_k > 0$ such that $\lim_{k \to \infty} r_k = 0$. Let $B_k(x) = \{ y \in X : \rho(x, y) < r_k \}$. Then the set of $x \in X$ such that
		$$\frac{1}{\mu(B_k(x))} \int_{B_k(x)} \frac{1}{k} \sum_{i = 0}^{k - 1} T^i f \mathrm{d} \mu \stackrel{k \to \infty}{\to } \int f \mathrm{d} \mu$$
		for all $f \in C(X)$ is of full measure.
	\end{Thm}
	
	\begin{proof}
		First we will prove that for an arbitrary $f \in C(X)$, the set of all $x \in X$ such that $\frac{1}{k} \sum_{i = 0}^{k - 1} \alpha_{B_k(x)} \left( T^i f \right) \stackrel{k \to \infty}{\to } \int f \mathrm{d} \mu$ is of full measure. Fix $\epsilon > 0$, and choose $\delta > 0$ such that $\rho(x, y) < \delta \Rightarrow |f(x) - f(y)| < \frac{\epsilon}{2}$ (where we invoke the uniform continuity of $f$). Choose $K_1 \in \mathbb{N}$ such that $r_{K_1} < \delta$. Then if $k \geq K_1 , i \in [0, k - 1]$, we have that $y \in B_k(x) \Rightarrow \rho(x, y) < \delta \Rightarrow \rho\left( T^i x, T^i y \right) < \delta$. Let $x \in \operatorname{supp}(\mu), k \geq K_1$. Then
		\begin{align*}
			\left| T^i f(x) - \alpha_{B_k(x)} \left( T^i f \right) \right|	& = \left| T^i f(x) - \frac{1}{\mu(B_k(x))} \int_{B_k(x)} T^i f \mathrm{d} \mu \right| \\
			& = \left| T^i f(x) - \frac{1}{\mu\left(T^i B_k(x)\right)} \int_{T^i B_k(x)} f \mathrm{d} \mu \right| \\
			& = \left| \frac{1}{\mu \left(T^i B_k(x) \right)} \int_{T^i B_k(x)} \left( T^i f(x) - f \right) \mathrm{d} \mu \right| \\
			& \leq \frac{1}{\mu \left(T^i B_k(x) \right)} \int_{T^i B_k(x)} \left| T^i f(x) - f \right| \mathrm{d} \mu \\
			& \leq \frac{\epsilon}{2} .
		\end{align*}
		
		Let $x \in X \cap \operatorname{supp}(\mu)$ such that $\frac{1}{k} \sum_{i = 0}^{k - 1} T^i f(x) \stackrel{k \to \infty}{\to} \int f \mathrm{d} \mu$. Choose $K_2 \in \mathbb{N}$ such that
		$$k \geq K_2 \Rightarrow \left| \int f \mathrm{d} \mu - \frac{1}{k} \sum_{i = 0}^{k - 1} T^i f(x) \right| < \frac{\epsilon}{2}.$$
		Then if $k \geq \max \{ K_1, K_2 \}$, then
		\begin{align*}
			& \left| \int f \mathrm{d} \mu - \frac{1}{k} \sum_{i = 0}^{k - 1} \alpha_{B_k(x)} \left(T^i f\right) \right|	\\
			& \leq \left| \int f \mathrm{d} \mu - \frac{1}{k} \sum_{i = 0}^{k - 1} T^i f(x) \right| + \left| \frac{1}{k} \sum_{i = 0}^{k - 1} T^i f (x) - \frac{1}{k} \sum_{i = 0}^{k - 1} \alpha_{B_k(x)} \left( T^i \right) \right| \\
			& \leq \frac{\epsilon}{2} + \frac{\epsilon}{2} \\
			& = \epsilon .
		\end{align*}
		The Birkhoff Ergodic Theorem then tells us that the set \linebreak
		$ \left\{ x \in X : \frac{1}{k} \sum_{i = 0}^{k - 1} T^i f(x) \stackrel{k \to \infty}{\to} \int f \mathrm{d} \mu \right\}$
		is of full measure, and so we can intersect it with the support of $\mu$ to get another set of full measure.
		
		We can now use an argument almost identical to that used in the proof of Theorem \ref{Non-expansinve random cylinders} to prove this present theorem. Let $\{ f_n \}_{n \in \mathbb{N}}$ be a countable set in $C(X)$ such that $\overline{\operatorname{span}}\{ f_n \}_{n \in \mathbb{N}} = C(X)$, where the closure is taken in the uniform norm on $C(X)$. Let
		$$S_n = \left\{ x \in X \cap \operatorname{supp}(\mu) : \frac{1}{k} \sum_{i = 0}^{k - 1} \alpha_{B_k(x)} (f_n) \to \int f_n \mathrm{d} \mu \right\} .$$
		As we have already shown, each $S_n$ is of full measure, and thus so is $\bigcap_{n \in \mathbb{N}} S_n$. From here, appealing to the fact that these $\{ f_n \}_{n \in \mathbb{N}}$ generate $C(X)$, we can prove the present theorem.
	\end{proof}
	
	\begin{Rmk}
		Assuming that $\mu(\{x\}) = 0$ for all $x \in X$, then $\mu(B_k(x)) \to 0$ for all $x \in X$.
	\end{Rmk}
	
	\begin{Ex} Theorem \ref{Non-expansive random balls} ceases to be true if we drop the hypothesis that our system is ergodic. Let $X = (X, \rho)$ be a compact metric space, and let $T : X \to X$ be the identity map $T = \operatorname{id}_X$ on $X$. Let $\mu$ be any non-atomic Borel probability measure $\mu$ on $X$ (which is automatically $\operatorname{id}_X$-invariant) that is strictly positive. Fix $x_0 \in X$ and let $f(x) = \rho(x, x_0)$. Let $B_k(x_0) = \{ x \in X : \rho(x, x_0) < 1/k \}$.
		
		We claim that $\int f \mathrm{d} \mu > 0$, but $$\alpha_{B_k(x_0)}\left( \frac{1}{k} \sum_{i = 0}^{k - 1} T^i f \right) \stackrel{k \to \infty}{\to} 0 .$$
		First, we observe that $\mu(B_k(x_0)) > 0$ for all $k \in \mathbb{N}$, since $x_0 \in \operatorname{supp}(\mu)$. However, since $\bigcap_{k = 1}^{\infty} B_k(x_0) = \{x_0\}$, we know that $\mu(B_k(x_0)) \to 0$. Therefore there exists $K \in \mathbb{N}$ such that $0 < \mu (B_K(x_0)) < \mu(B_1(x_0))$. Since $f$ is a nonnegative function, we can then conclude that
		\begin{align*}
			\int f \mathrm{d} \mu	& \geq \int_{B_1(x_0) \setminus B_K(x_0)} f \mathrm{d} \mu \\
			& \geq \mu(B_1(x_0) \setminus B_K(x_0)) \frac{1}{K} \\
			& > 0 .
		\end{align*}
		
		Then
		\begin{align*}
			\alpha_{B_k(x_0)} \left( \frac{1}{k} \sum_{i = 0}^{k - 1} T^i f \right)	& = \alpha_{B_k(x_0)} \left( \frac{1}{k} \sum_{i = 0}^{k - 1} f \right) \\
			& = \alpha_{B_k(x_0)}(f) \\
			& = \frac{1}{\mu(B_k(x_0))} \int_{B_k(x_0)} f \mathrm{d} \mu .
		\end{align*}
		
		However, we can also say that $\left| \alpha_{B_k(x_0)}(f) \right| \leq 1/k$, since
		\begin{align*}
			\left|\alpha_{B_k(x_0)}(f) \right|	& = \left| \frac{1}{\mu(B_k(x_0))} \int_{B_k(x_0)} f \mathrm{d} \mu \right| \\
			& \leq \frac{1}{\mu(B_k(x_0))} \int_{B_k(x_0)} |f| \mathrm{d} \mu \\
			& \leq \frac{1}{\mu(B_k(x_0))} \int_{B_k(x_0)} \frac{1}{k} \mathrm{d} \mu \\
			& = \frac{1}{k} .
		\end{align*}
		Thus $\alpha_{B_k(x_0)} \left( \frac{1}{k} \sum_{i = 0}^{k - 1} T^i f \right) \stackrel{k \to \infty}{\to} 0 \neq \int f \mathrm{d} \mu$.
		
		Thus for \emph{every} $x_0 \in X$ exists $f_{x_0} \in C(X)$ such that
		$$\limsup_{k \to \infty} \left| \int f_{x_0} \mathrm{d} \mu - \alpha_{B_k(x_0)} \left( \frac{1}{k} \sum_{i = 0}^{k - 1} T^i f_{x_0} \right) \right| > 0 .$$
		
		This example highlights how the assumption that $T$ is ergodic is pulling some amount of weight.
	\end{Ex}

	\begin{Rmk}
	In this section, as well as in Sections \ref{Lipschitz} and \ref{Bernoulli shifts and probability}, we focus on continuous functions $f \in C(X)$. Our reason for this is that we can study these $f$ in relation to the topological properties of $(X, T)$.
	\end{Rmk}
	
	\section{Lipschitz maps and subshifts}\label{Lipschitz}
	
	Let us consider a compact \emph{pseudo}metric space $X = (X, p)$, and $T : X \to X$ a map that is Lipschitz of constant $L > 1$, i.e. $p(Tx, Ty) \leq L \cdot p(x, y)$. Recall that a pseudometric is distinguished from a metric by the fact we do not assume that a pseudometric distinguishes points, i.e. we do not assume that $p(x, y) = 0 \Rightarrow x = y$. Suppose that $(X, \mathcal{B}, \mu, T)$ is of finite entropy, and thus admits a generator $\mathcal{E}$. Suppose further that for almost all $x \in X$ exists a constant $\gamma = \gamma_x \in [1, \infty)$ such that
	$$\operatorname{diam}(C_k(x)) \leq \gamma \cdot L^{-k} \; (\forall k \in \mathbb{N}) .$$
	
	We pause to remark on two points. The first is that our consideration of pseudometric spaces is not generality for generality's sake. As we will see later in this section, this consideration of pseudometric spaces will be useful for studying certain metric spaces. The second is that this class of examples is not a direct generalization of the class considered in Section \ref{Non-expansive}. Though every $1$-Lipschitz map is of course Lipschitz for every constant $L > 1$, our condition on $\operatorname{diam}(C_k(x))$ is stronger here, since we ask not just that $\operatorname{diam}(C_k(x))$ go to $0$, but that it do so exponentially.
	
	Since we are working in the slightly unorthodox setting of pseudometric spaces rather than metric spaces, we will prove that one of the strong properties of compact metric spaces is also true of compact \textit{pseudo}metric spaces, namely that every continuous function is uniformly continuous. The proof is essentially identical to the "textbook" argument for compact metric spaces. We doubt this is a new result, but we could not find a reference for it, so we prove it here.
	
	\begin{Lem}
		Let $(X, p)$ be a compact pseudometric space. Then every continuous function $f : X \to \mathbb{C}$ is uniformly continuous.
	\end{Lem}
	
	\begin{proof}
		Fix $\epsilon > 0$. Then for every $x \in X$ exists $\delta_x > 0$ such that $p(x, y) < \delta_x \Rightarrow |f(x) - f(y)| < \epsilon$. Then the family $\mathcal{U} = \left\{ B \left( x, \frac{\delta_x}{2} \right) \right\}_{x \in X}$ is an open cover of $X$, so there exists a finite subcover $\mathcal{U}' = \left\{ B \left( x_j, \frac{\delta_{x_j}}{2} \right) \right\}_{j = 1}^n$ of $X$.
		
		Let $\delta' = \min_{1 \leq j \leq n} \frac{\delta_{x_j}}{2}$, and suppose that $x, y \in X$ such that $p(x, y) < \delta'$. Then there exists $x_j \in X$ such that $p(x, x_j) < \frac{\delta_{x_j}}{2}$, since $\mathcal{U}'$ is a cover of $X$. Then
		\begin{align*}
			p(x_j, y)	& \leq p(x_j, x) + p(x, y) \\
			& < \frac{\delta_{x_j}}{2} + \frac{\delta_{x_j}}{2} \\
			& = \delta_{x_j} .
		\end{align*}
		Therefore $p(x_j, x) < \frac{\delta_{x_j}}{2} < \delta_{x_j}, p(x_j, y) < \delta_{x_j}$, so $|f(x) - f(x_j)| < \frac{\epsilon}{2} , |f(y) - f(x_j)| < \frac{\epsilon}{2}$. 
		Thus
		\begin{align*}
			|f(x) - f(y)|	& \leq |f(x) - f(x_j)| + |f(x_j) - f(y)| \\
			& < \frac{\epsilon}{2} + \frac{\epsilon}{2} \\
			& = \epsilon .
		\end{align*}
		
		Therefore $\delta' > 0$ is such that $p(x, y) < \delta' \Rightarrow |f(x) - f(y)| < \epsilon$. Thus we have shown that $f$ is uniformly continuous.
	\end{proof}
	
	Now we are able to both state and prove the first main result of this section.
	
	\begin{Prop}\label{Random Lipschitz Cylinders}
		Let $(X, p)$ be a compact pseudometric space, and let $T : X \to X$ be an $L$-Lipschitz homeomorphism on $X$ with respect to $p$, where $L > 1$. Suppose $\mu$ is a regular Borel probability measure on $X$ such that $T$ is ergodic with respect to $\mu$. Let $\mathcal{E} = \{ E_d \}_{d \in \mathcal{D}}$ be a generator of $(X, T)$ such that for almost all $x \in X$ exists $\gamma_x \in \mathbb{R}$ such that $\operatorname{diam}(C_k(x)) \leq \gamma_x \cdot L^{-k}$ for all $k \in \mathbb{N}$. Fix $f \in C(X)$. Then
		$$\frac{1}{k} \sum_{i = 0}^{k - 1} \alpha_{C_k(x)} \left( T^i f \right) \stackrel{k \to \infty}{\to } \int f \mathrm{d} \mu$$
		for almost all $x \in X$.
	\end{Prop}
	
	\begin{proof}
	Our goal is to show that for every $\epsilon > 0$, there exists some $K \in \mathbb{N}$ such that if $k \geq K$, we have
	\begin{align*}
		\left| \int f \mathrm{d} \mu - \frac{1}{k} \sum_{i = 0}^{k - 1} \alpha_{C_k(x)} \left( T^i f \right) \right|	\\
		\leq \left| \int f \mathrm{d} \mu - \frac{1}{k} \sum_{i = 0}^{k - 1} \left( T^i f \right)(x) \right| + \left| \frac{1}{k} \sum_{i = 0}^{k - 1} \left( \left( T^i f \right)(x) - \alpha_{C_k(x)} \left( T^i f \right) \right) \right| \\
		\leq \left| \int f \mathrm{d} \mu - \frac{1}{k} \sum_{i = 0}^{k - 1} \left( T^i f \right)(x) \right| + \frac{1}{k} \sum_{i = 0}^{k - 1} \left| \left( T^i f \right)(x) - \alpha_{C_k(x)} \left( T^i f \right) \right| \\
		\leq \epsilon .
	\end{align*}
	We will accomplish this by bounding the terms
	$$\left| \int f \mathrm{d} \mu - \frac{1}{k} \sum_{i = 0}^{k - 1} \left( T^i f \right)(x) \right| , \; \frac{1}{k} \sum_{i = 0}^{k - 1} \left| \left( T^i f \right)(x) - \alpha_{C_k(x)} \left( T^i f \right) \right| $$
	by $\epsilon$.
	
		We will start with bounding the latter term. We claim that if $x \in X$ such that $\mu(C_k(x)) > 0, \operatorname{diam}(C_k(x)) \leq \gamma_x \cdot L^{-k}$ for all $k \in \mathbb{N}$, then for every $\epsilon > 0$, there exists $K_1 \in \mathbb{N}$ such that
		$$k \geq K_1 \Rightarrow \frac{1}{k} \sum_{i = 0}^{k - 1} \left| \left( T^i f \right) (x) - \alpha_{C_k(x)} \left( T^i f \right) \right| < \frac{\epsilon}{2} .$$
		To prove this, choose $\delta > 0$ such that $p(y, z) < \delta \Rightarrow |f(y) - f(z)| < \frac{\epsilon}{4}$. Let $\kappa \in \mathbb{N}$ such that $\gamma_x \cdot L^{-\kappa} < \delta$. Then if $k > \kappa$, then
		\begin{align*}
		\frac{1}{k} \sum_{i = 0}^{k - 1} \left| \left( T^i f \right)(x) - \alpha_{C_k(x)} \left( T^i f \right) \right|	& = \frac{1}{k} \left[ \sum_{i = 0}^{k - \kappa} \left| \left( T^i f \right)(x) - \alpha_{C_k(x)} \left( T^i f \right) \right| \right] \\
			& + \frac{1}{k} \left[ \sum_{k - \kappa + 1}^{k - 1} \left| \left( T^i f \right)(x) - \alpha_{C_k(x)} \left( T^i f \right) \right| \right] .
		\end{align*}
		We will estimate these two terms separately, bounding each by $\frac{\epsilon}{4}$. Beginning with the former, we observe that if $x, y \in C_k(x)$, then
		$$p \left( T^{i} x, T^i y \right) \leq L^{i} p(x, y) \leq L^i \cdot \gamma_x \cdot L^{-k} = \gamma_{x} \cdot L^{i - k} .$$
		In particular, this means that if $i - k \leq - \kappa$, then $\left| \left( T^i f \right)(x) - f(z) \right| < \frac{\epsilon}{4}$ for all $z = T^i y \in T^i C_{k}(x)$, so
		\begin{align*}
			& \frac{1}{k} \left[ \sum_{i = 0}^{k - \kappa} \left| \left( T^i f \right)(x) - \alpha_{C_k(x)} \left( T^i f \right) \right| \right]	\\
			& = \frac{1}{k} \left[ \sum_{i = 0}^{k - \kappa} \left| \frac{1}{\mu(C_k(x))} \int_{C_k(x)} \left( \left( T^i f \right)(x) \right) - T^i f \mathrm{d} \mu \right| \right] \\
			& = \frac{1}{k} \left[ \sum_{i = 0}^{k - \kappa} \frac{1}{\mu\left(T^i C_k(x)\right)} \int_{T^i C_k(x)} \left| \left( T^i f \right) (x) - f \right| \mathrm{d} \mu \right] \\
			& \leq \frac{1}{k} \left[ \sum_{i = 0}^{k - \kappa} \frac{1}{\mu\left(T^i C_k(x)\right)} \int_{T^i C_k(x)} \frac{\epsilon}{4} \mathrm{d} \mu \right] \\
			& = \frac{k - \kappa + 1}{k} \frac{\epsilon}{4} \\
			& \leq \frac{\epsilon}{4} .
		\end{align*}
		On the other hand, we can estimate
		$$\frac{1}{k} \left[ \sum_{k - \kappa + 1}^{k - 1} \left| \left( T^i f \right)(x) - \alpha_{C_k(x)} \left( T^i f \right) \right| \right]  \leq \frac{2 \kappa}{k} \| f \| .$$
		Choose $K_1 > \kappa$ such that $\frac{2 \kappa \left\| f \right\|_\infty}{K_1} < \frac{\epsilon}{4}$. Then if $k \geq K_1$, we have
		\begin{align*}
			\frac{1}{k} \sum_{i = 0}^{k - 1} \left| \left( T^i f \right)(x) - \alpha_{C_k(x)} \left( T^i f \right) \right|	& = \frac{1}{k} \left[ \sum_{i = 0}^{k - \kappa} \left| \left( T^i f \right)(x) - \alpha_{C_k(x)} \left( T^i f \right) \right| \right] \\
			& + \frac{1}{k} \left[ \sum_{k - \kappa + 1}^{k - 1} \left| \left( T^i f \right)(x) - \alpha_{C_k(x)} \left( T^i f \right) \right| \right] \\
			& \leq \frac{\epsilon}{4} + \frac{\epsilon}{4} \\
			& = \frac{\epsilon}{2} .
		\end{align*}
	
		Now suppose further that $x \in X$ is such that $\frac{1}{k} \sum_{i = 0}^{k - 1} \left( T^i f \right)(x) \stackrel{k \to \infty}{\to} \int f \mathrm{d} \mu$. Choose $K_2 \in \mathbb{N}$ such that $k \geq K_2 \Rightarrow \left| \int f \mathrm{d} \mu - \frac{1}{k} \sum_{i = 0}^{k - 1} \left( T^i f \right)(x) \right| < \frac{\epsilon}{2}$. Then if $k \geq \max \{ K_1, K_2 \}$, then we have
		\begin{align*}
			\frac{1}{k} \sum_{i = 0}^{k - 1} \left| \left( T^i f \right)(x) - \alpha_{C_k(x)} \left( T^i f \right) \right|	& = \frac{1}{k} \left[ \sum_{i = 0}^{k - \kappa} \left| \left( T^i f \right)(x) - \alpha_{C_k(x)} \left( T^i f \right) \right| \right] \\
			& + \frac{1}{k} \left[ \sum_{k - \kappa + 1}^{k - 1} \left| \left( T^i f \right)(x) - \alpha_{C_k(x)} \left( T^i f \right) \right| \right] \\
			& \leq \frac{\epsilon}{2} + \frac{\epsilon}{2} \\
			& = \epsilon .
		\end{align*}
	Since the set of $x \in X$ for which this calculation could be performed is of full measure, the proposition follows.
	\end{proof}
	
	From here, we get the following corollary.
	
	\begin{Cor}
		Let $(X, \rho)$ be a compact metric space, and let $T : X \to X$ be an $L$-Lipschitz homeomorphism on $X$ with respect to $\rho$, where $L > 1$. Suppose $\mu$ is a regular Borel probability measure on $X$ such that $T$ is ergodic with respect to $\mu$. Let $\mathcal{E} = \{ E_d \}_{d \in \mathcal{D}}$ be a generator of $(X, T)$ such that for almost all $x \in X$ exists $\gamma_x \in \mathbb{R}$ such that $\operatorname{diam}(C_k(x)) \leq \gamma_x \cdot L^{-k}$ for all $k \in \mathbb{N}$. Then the set of $x \in X$ such that
		$$\frac{1}{k} \sum_{i = 0}^{k - 1} \alpha_{C_k(x)} \left( T^i f \right) \stackrel{k \to \infty}{\to } \int f \mathrm{d} \mu$$
		for all $f \in C(X)$ is of full measure.
	\end{Cor}
	
	\begin{proof}
		Let $\{ f_n \}_{n \in \mathbb{N}}$ be a countable set in $C(X)$ such that \linebreak$C(X) = \overline{\operatorname{span}} \{ f_n \}_{n \in \mathbb{N}}$. By the previous result, we can extrapolate that the set of $x \in X$ such that $\frac{1}{k} \sum_{i = 0}^{k - 1} \alpha_{C_k(x)} \left( T^i f_n \right) \stackrel{k \to \infty}{\to } \int f_n \mathrm{d} \mu$ is of full measure. We can then extend to all of $C(X)$ in the same manner as we did in the proof of Theorem \ref{Non-expansinve random cylinders}.
	\end{proof}
	
	\subsection{Two-sided subshifts and systems of finite entropy}
	
	This brings us to the matter of (two-sided) subshifts. Let $\mathcal{D}$ be a finite discrete set, and let $T : \mathcal{D}^\mathbb{Z} \to \mathcal{D}^\mathbb{Z}$ be the map $(Tx)_n = x_{n + 1}$, called the \emph{left shift}. We call $X \subseteq \mathcal{D}^\mathbb{Z}$ a \emph{subshift} if $X$ is compact and $T X = X$. Assume that $\mu$ is a Borel probability measure on $X$ with respect to which $T$ is ergodic.
	
	In a shift space, we will always take our generator to be the family $\mathcal{E} = \{ E_d \}_{d \in \mathcal{D}}$ of sets $E_d = \{ x \in X : x_0 = d \} , d \in \mathcal{D}$. We claim that for almost all $x \in X$, we have
	$$\lim_{k \to \infty} \frac{1}{k} \sum_{i = 0}^{k - 1} \alpha_{C_k(x)} \left( T^i f \right) = \int f \mathrm{d} \mu$$
	for all $f \in C(X)$. First, we want to establish the following lemma.
	
	\begin{Lem}\label{Cylinders dense}
		Let $(X, \mathcal{F}, \mu, T)$ be a subshift, where $X \subset \mathcal{D}^\mathbb{Z}$. The family
		$$\mathcal{F} = \left\{ T^n \chi_{[a_0, a_1, \ldots, a_{\ell - 1}]} : {(a_0, a_1, \ldots, a_{\ell - 1}) \in \mathcal{D}^\ell , \ell \in \mathbb{N}, i \in \mathbb{Z}} \right\}$$
		generates $C(X)$ in the sense that its span is dense in $C(X)$ with respect to the uniform norm.
	\end{Lem}
	
	\begin{proof}
		We claim that every $f \in C(X)$ can be approximated uniformly by elements of $\operatorname{span} \mathcal{F}$. We will begin by demonstrating the result for real $f \in C(X)$, then extrapolate the result to all complex-valued $f \in C(X)$.
		
		For $\ell \in \mathbb{N}$, set
		\begin{align*}
			& A(a_{-\ell + 1}, a_{-\ell + 1}, \ldots, a_{-1}, a_0, a_1, \ldots, a_{\ell - 1}, a_{\ell - 1}) \\
			& = \{ x \in X : x_{j} = a_j \; \forall j \in [-\ell + 1, \ell - 1] \} \\
			& = T^{-\ell + 1} [a_{-\ell + 1}, a_{-\ell + 2}, \ldots, a_{-1}, a_0, a_1, \ldots, a_{\ell - 1}, a_{\ell - 1}]
		\end{align*}
		and let
		\begin{align*}
			g_\ell	& = \sum_{\mathbf{a} \in \mathcal{D}^{2 \ell - 1}} \min \left\{ f(y) : y \in A(\mathbf{a}) \right\} \chi_{A(\mathbf{a})} \\
			& = \sum_{\mathbf{a} \in \mathcal{D}^{2 \ell - 1}} \min \left\{ f(y) : y \in T^{-\ell +1} [\mathbf{a}] \right\} \chi_{T^{-\ell} [\mathbf{a}]} \\
			& \in \operatorname{span} \mathcal{F} .
		\end{align*}
		We claim that $g_\ell \to f$ uniformly. The sequence $\ell \mapsto g_\ell$ is monotonic increasing. Moreover, we claim that it converges pointwise to $f$. To see this, let $x \in X$, and consider $g_\ell(x)$. Fix $\epsilon > 0$. Then for each $\ell \in \mathbb{N}$ exists $y^{(\ell)} \in X$ such that $g_\ell(x) = f\left( y^{(\ell)} \right)$. However, since $y_j^{(\ell)} = x_j$ for all $j \in [-\ell + 1, \ell - 1]$, we can conclude that $y^{(\ell)} \to x$, and so by continuity of $f$, we can conclude that $g_\ell(x) = f \left( y^{(\ell)} \right) \to f(x)$. Thus $g_\ell \nearrow f$ pointwise. Dini's Theorem then gives us uniform convergence. Therefore, if $f \in C(X)$ is real-valued, then $f \in \overline{\operatorname{span}} \mathcal{F}$. On the other hand, any complex-valued function $f \in C(X)$ can be expressed as the sum of its real and imaginary parts, and we can apply this argument to both of those parts separately.
	\end{proof}
	
	\begin{Thm}\label{Metric result for subshifts}
		Let $X \subseteq \mathcal{D}^\mathbb{Z}$ be a subshift, and let $\mu$ be a Borel probability measure on $X$ with respect to which the left shift $T$ is ergodic. Then the set of all $x \in X$ such that
		$$\frac{1}{k} \sum_{i = 0}^{k - 1} \alpha_{C_k(x)} \left( T^i f \right) \to \int f \mathrm{d} \mu$$
		for all $f \in C(X)$ is of full measure.
	\end{Thm}
	
	\begin{proof}
		Our first step is to show that
		$$\frac{1}{k} \sum_{i = 0}^{k - 1} \alpha_{C_k(x)} \left( T^i \chi_{[a_0, a_1, \ldots, a_{\ell - 1}]} \right) \to \int \chi_{[a_0, a_1, \ldots, a_{\ell - 1}]} \mathrm{d} \mu$$
		for all finite strings $\mathbf{a} = (a_0, a_1, \ldots, a_{\ell - 1}) \in \mathcal{D}^\ell$. Let $p$ be the pseudometric on $X$ given by
		$p(x, y) = 2^{- \min \{ n \geq 0 : x_n \neq y_n \}} ,$
		where $\min (\emptyset) = + \infty$ and $2^{-\infty} = 0$.
		
		We claim that the function $\chi_{[\mathbf{a}]}$ is continuous with respect to the topology of $p$, and that $(X, \mathcal{B}, \mu, T)$ satisfies the hypotheses of Proposition \ref{Random Lipschitz Cylinders} for $L = 2$. A straightforward calculation shows that $T$ is $2$-Lipschitz and that $\operatorname{diam} \left([\mathbf{a}]\right) \leq 2 \cdot 2^{-\ell}$ for all $\ell \in \mathbb{N}, \mathbf{a} \in \mathcal{D}^\ell$. Therefore, if
		$$R_\mathbf{a} = \left\{ x \in X : \frac{1}{k} \sum_{i = 0}^{k - 1} \alpha_{C_k(x)} \left( T^i \chi_{[a_0, a_1, \ldots, a_{\ell - 1}]} \right) \to \int \chi_{[a_0, a_1, \ldots, a_{\ell - 1}]} \mathrm{d} \mu \right\},$$
		then $\mu(R_\mathbf{a}) = 1$ for all $\mathbf{a} \in \bigcup_{\ell = 1}^\infty \mathcal{D}^\ell$, and so $R = \bigcap_{\mathbf{a} \in \bigcup_{\ell = 1}^\infty \mathcal{D}^\ell} R_\mathbf{a}$ is of full measure. We now claim that if $\frac{1}{k} \sum_{i = 0}^{k - 1} \alpha_{C_k(x)} \left( T^i \chi_{[a_0, a_1, \ldots, a_{\ell - 1}]} \right) \to \int \chi_{[a_0, a_1, \ldots, a_{\ell - 1}]} \mathrm{d} \mu$, then $$\frac{1}{k} \sum_{i = 0}^{k - 1} \alpha_{C_k(x)} \left( T^i T^n \chi_{[a_0, a_1, \ldots, a_{\ell - 1}]} \right) \to \int \chi_{[a_0, a_1, \ldots, a_{\ell - 1}]} \mathrm{d} \mu $$
		for all $n \in \mathbb{Z}$.
		
		It will suffice to prove the result for $n = \pm 1$ and extend to all $n \in \mathbb{Z}$ by induction. To prove the claim for $n = 1$, we observe that
		\begin{align*}
			\left( \frac{1}{k} \sum_{i = 0}^{k - 1} \alpha_{C_k(x)} \left( T^i (T f) \right) \right) - \left( \frac{1}{k} \sum_{i = 0}^{k - 1} \alpha_{C_k(x)} \left( T^i f \right) \right)	& = \frac{1}{k} \alpha_{C_k(x)} \left( T^k f - f \right) \\
			\Rightarrow \left| \left( \frac{1}{k} \sum_{i = 0}^{k - 1} \alpha_{C_k(x)} \left( T^i (T f) \right) \right) - \left( \frac{1}{k} \sum_{i = 0}^{k - 1} \alpha_{C_k(x)} \left( T^i f \right) \right) \right|	& \leq \frac{2}{k} \| f \|_\infty \\
			& \to 0 .
		\end{align*}
		A similar calculation tells us that
		$$\left| \left( \frac{1}{k} \sum_{i = 0}^{k - 1} \alpha_{C_k(x)} \left( T^i \left( T^{-1} f \right) \right) \right) - \left( \frac{1}{k} \sum_{i = 0}^{k - 1} \alpha_{C_k(x)} \left( T^i f \right) \right) \right| \leq \frac{2}{k} \| f \|_\infty \to 0 ,$$
		verifying the claim for $n = - 1$.
		Thus if $\frac{1}{k} \sum_{i = 0}^{k - 1} \alpha_{C_k(x)} \left( T^i f \right) \to \int f \mathrm{d} \mu$, then a straightforward induction argument will show that
		$$\frac{1}{k} \sum_{i = 0}^{k - 1} \alpha_{C_k(x)} \left( T^i \left( T^n f \right) \right) \to \int f \mathrm{d} \mu = \int T^n f \mathrm{d} \mu $$
		for all $n \in \mathbb{Z}$.
		
		In particular, this means that if $x \in R$, then $\frac{1}{k} \sum_{i = 0}^{k - 1} \alpha_{C_k(x)} \left( T^i f \right) \to \int f \mathrm{d} \mu$ for all $f \in \mathcal{F}$. Since the span of $\mathcal{F}$ is dense in $C(X)$, this means that if $x \in R$, then
		$$\frac{1}{k} \sum_{i = 0}^{k - 1} \alpha_{C_k(x)} \left( T^i f \right) \to \int f \mathrm{d} \mu$$
		for all $f \in C(X)$.
	\end{proof}
	
	We turn now to apply Theorem \ref{Metric result for subshifts} to a slightly broader context. Let $(Y, \mathcal{A}, \nu, S)$ be an invertible ergodic system with finite entropy. Then the system admits a finite generator $\mathcal{E} = \{E_d\}_{d \in \mathcal{D}}$. For each $y \in Y, i \in \mathbb{Z}$, let $e_i(y) \in \mathcal{D}$ be the element of $\mathcal{D}$ such that $y \in S^{-i} E_{e_i(y)}$, or equivalently such that $S^i y \in E_{e_i(y)}$. Define the \emph{$k$-length cylinder} corresponding to $y$ by
	$$F_k(y) = \bigcap_{i = 0}^{k - 1} S^{-i} E_{e_i(y)} .$$
	
	We define a map $\phi : Y \to \mathcal{D}^\mathbb{Z}$ by
	$$\phi(y) = (e_i(y))_{i \in \mathbb{Z}} .$$
	We call this map $\phi$ the \emph{itinerary map} on $Y$ induced by $\mathcal{E}$. Let $T$ be the standard left shift on $\mathcal{D}^\mathbb{Z}$. The itinerary map commutes with the left shift in the sense that the following diagram commutes:
	$$
	\begin{tikzcd}
		Y \arrow[d, "\phi"] \arrow[r, "S"] & Y \arrow[d, "\phi"] \\
		\mathcal{D}^\mathbb{Z} \arrow[r, "T"] & \mathcal{D}^\mathbb{Z}
	\end{tikzcd}
	$$
	We can now state the following corollary.
	
	\begin{Cor}
		Let $(Y, \mathcal{A}, \nu, S)$ be an invertible ergodic system with finite entropy and finite generator $\mathcal{E} = \{E_d\}_{d \in \mathcal{D}}$. Let $\mathbb{A} \subseteq L^\infty(Y, \nu)$ be the subspace
		$$\mathbb{A} = \overline{\operatorname{span}} \left\{ S^n \chi_{\bigcap_{j = 0}^{\ell - 1} S^{-j} E_{d_j}} : n \in \mathbb{Z}, \ell \in \mathbb{N}, d_j \in \mathcal{D} \right\} .$$
		Then the set of $y \in Y$ such that
		$$\frac{1}{k} \sum_{i = 0}^{k - 1} \frac{1}{\mu(F_k(y))} \int_{F_k(y)} S^i g \mathrm{d} \nu \to \int g \mathrm{d} \nu$$
		for all $g \in \mathbb{A}$ is of full measure.
	\end{Cor}
	
	\begin{proof}
		Endow $\mathcal{D}^\mathbb{Z}$ with the pushforward measure $\mu(B) = \nu \left( \phi^{-1} B \right)$. Since $\phi^{-1}[d] = E_d \in \mathcal{A}$ for all $d \in \mathcal{D}$, we know that $\mu$ is Borel. We also observe that $F_k(y) = \phi^{-1} C_k(\phi(y))$. Consider $f = \chi_{\bigcap_{i = 0}^{k - 1} E_{d_i}} = \chi_{\phi^{-1} [d_0, d_1, \ldots, d_{k - 1}] }$. Let $B \subseteq \mathcal{D}^\mathbb{Z}$ be the set of all $x \in X$ such that
		$$\frac{1}{k} \sum_{i = 0}^{k - 1} \alpha_{C_k(x)} \left( T^i f \right) \to \int f \mathrm{d} \mu$$
		for all $f \in C(X)$, which we know by the previous theorem to be of full measure in $X$, and let $A = \phi^{-1} B$. Then if $y \in A$, and $d_0, d_1, \ldots, d_{\ell - 1} \in \mathcal{D}$, then
		\begin{align*}
			& \frac{1}{k} \sum_{i = 0}^{k - 1} \frac{1}{\nu(F_k(y))} \int_{F_k(y)} S^i \chi_{\bigcap_{j = 0}^{\ell - 1} S^{-j} E_{d_j}} \mathrm{d} \nu \\
			& = \frac{1}{k} \sum_{i = 0}^{k - 1} \frac{1}{\nu(F_k(y))} \nu \left( F_k(y) \cap S^{-i} \bigcap_{j = 0}^{\ell - 1} S^{-j} E_{d_j} \right) \\
			& = \frac{1}{k} \sum_{i = 0}^{k - 1} \frac{1}{\nu \left( \phi^{-1} C_k(\phi(y)) \right)} \nu \left(  \phi^{-1} \left( C_k(\phi(y)) \cap T^{-i} [d_0, d_1, \ldots, d_{\ell - 1}] \right) \right) \\
			& = \frac{1}{k} \sum_{i = 0}^{k - 1} \frac{1}{\mu (C_k(\phi(y))) } \mu \left( C_k(\phi(y)) \cap T^{-i} [d_0, d_1, \ldots, d_{\ell - 1}] \right) \\
			& = \frac{1}{k} \sum_{i = 0}^{k - 1} \frac{1}{\mu (C_k(\phi(y))) } \int_{C_k(\phi(y)} T^{i} \chi_{[d_0, d_1, \ldots, d_{\ell - 1}]} \mathrm{d} \mu \\
			& \to \int \chi_{[d_0, d_1, \ldots, d_{\ell - 1}]} \mathrm{d} \mu \\
			& = \mu([d_0, d_1, \ldots, d_{\ell - 1}]) \\
			& = \nu \left( \bigcap_{j = 0}^{\ell - 1} S^{-j} E_{d_j} \right) \\
			& = \int \chi_{\bigcap_{j = 0}^{\ell - 1} S^{-j} E_{d_j}} \mathrm{d} \nu ,
		\end{align*}
		since $\chi_{[d_0, d_1, \ldots, d_{\ell - 1}]} \in C \left( \mathcal{D}^\mathbb{Z} \right)$.
		By an argument similar to that employed in the proof of Theorem \ref{Metric result for subshifts}, we can extrapolate that if $y \in A$, then
		$$\frac{1}{k} \sum_{i = 0}^{k - 1} \frac{1}{\mu(F_k(y))} \int_{F_k(y)} S^i S^n \chi_{\bigcap_{j = 0}^{\ell - 1} S^{-j} E_{d_j}} \mathrm{d} \nu \to \int S^n \chi_{\bigcap_{j = 0}^{\ell - 1} S^{-j} E_{d_j}} \mathrm{d} \nu$$
		for $n \in \mathbb{Z}$. By density, it follows that if $g \in \mathbb{A}$, then \linebreak $\frac{1}{k} \sum_{i = 0}^{k - 1} \frac{1}{\mu(F_k(y))} \int_{F_k(y)} S^i g \mathrm{d} \nu \to \int g \mathrm{d} \nu$ for all $y \in A$, and $A$ is a set of full measure.
	\end{proof}
	
	\subsection{Pathological differentiation problems and relations to symbolic distributions}
	
	In Theorem \ref{Metric result for subshifts}, we demonstrated that
	$$\alpha_{C_k(x)} \left( \frac{1}{k} \sum_{i = 0}^{k - 1} T^i f \right) \stackrel{k \to \infty}{\to} \int f \mathrm{d} \mu \; (\forall f\in C(X)) $$
	for almost all $x \in X$. We take this opportunity to demonstrate that the "almost all" caveat is indispensable, as there can exist $x \in X$ for which $\alpha_{C_k(x)} \left( \frac{1}{k} \sum_{i = 0}^{k - 1} T^i f \right) \not \to \int f \mathrm{d} \mu$ for certain $f \in C(X)$. This is related to the shift not being uniquely ergodic, which we discussed in more detail in Section \ref{Topological STDs}. In fact, we even claim the sequence $\left( \alpha_{C_k(x)} \left( \frac{1}{k} \sum_{i = 0}^{k - 1} T^i f \right) \right)_{k = 1}^{\infty}$ can fail to be \emph{Cauchy} for certain pairs $(x, f) \in X \times C(X)$. 
	
	\begin{Thm}\label{Pathological Differentiation}
		Let $X = \mathcal{D}^\mathbb{Z}$ be a Bernoulli shift with symbol space \linebreak $\mathcal{D} = \{0, 1, \ldots, D - 1\}, D \geq 2$, a Borel probability measure $\mu$ such that $\mu([d]) \neq 0$ for all $d \in \mathcal{D}$. Let $f = \chi_{[0]}$, and left shift $T$. Then there exists an uncountable subset $S \subseteq X$ such that $x, y \in S \Rightarrow x_j = y_j \; \left( \forall j \leq 0 \right)$, and such that the sequence $\left( \alpha_{C_k(x)} \left( \frac{1}{k} \sum_{i = 0}^{k - 1} T^i f \right) \right)_{k = 1}^\infty$ is not Cauchy for all $x \in S$.
	\end{Thm}
	
	\begin{proof}
		We first compute $\alpha_{[x_0, x_1, \ldots, x_{k - 1}]} \left(T^i f \right)$ for $0 \leq i \leq k - 1$ as follows. We see
		\begin{align*}
			\alpha_{[x_0, x_1, \ldots, x_{k - 1}]} \left(T^i f \right)	& = \frac{1}{\mu([x_0, x_1, \ldots, x_{k - 1}])} \int_{[x_0, x_1, \ldots, x_{k - 1}]} T^i \chi_{[0]} \mathrm{d} \mu \\
			& = \frac{1}{\mu([x_0, x_1, \ldots, x_{k - 1}])} \int_{\bigcap_{j = 0}^{k - 1} T^{-j} [x_j]} \chi_{T^{-i} [0]} \mathrm{d} \mu \\
			& = \frac{1}{\mu([x_0, x_1, \ldots, x_{k - 1}])} \mu \left( \left( \bigcap_{j = 0}^{k - 1} T^{-j} [x_j] \right) \cap T^{-i}[0] \right) \\
			& = \delta(x_i, 0) ,
		\end{align*}
		where $\delta(\cdot, \cdot)$ refers here to the Kronecker delta. Thus if $x = (x_j)_{j \in \mathbb{Z}} \in X$, then
		\begin{align*}
			\alpha_{C_k(x)} \left( \frac{1}{k} \sum_{i = 0}^{k - 1} T^i f \right)	& = \frac{\# \{ i \in [0 , k - 1] : x_i = 0 \}}{k}	& (\dagger) \end{align*}
		
		The identity $(\dagger)$ implies that if there exists $x \in X$ such that \linebreak $\left( \alpha_{C_k(x)} \left( \frac{1}{k} \sum_{i = 0}^{k - 1} T^i f \right) \right)_{k = 1}^\infty$ is not Cauchy, then we can then build our set $S$. For $x, y \in X$, write $x \sim y$ if $x_j = y_j$ for all $j \leq 0$, and the set $\{ j \in \mathbb{N} : x_j \neq y_j \}$ has density $0$. This is an equivalence relation. We claim that if $x \sim y$, then $\left| \alpha_{C_k(x)} \left( \frac{1}{k} \sum_{i = 0}^{k - 1} T^i f \right) - \alpha_{C_k(y)} \left( \frac{1}{k} \sum_{i = 0}^{k - 1} T^i f \right) \right| \stackrel{k \to \infty}{\to} 0$. By $(\dagger)$, we know that
		\begin{align*}
			\left| \alpha_{C_k(x)} \left( \frac{1}{k} \sum_{i = 0}^{k - 1} T^i f \right) - \alpha_{C_k(y)} \left( \frac{1}{k} \sum_{i = 0}^{k - 1} T^i f \right) \right|	& = \left|\frac{1}{k} \sum_{i = 0}^{k - 1} (\delta(x_i, 0) - \delta(y_i, 0)) \right| \\
			& \leq \frac{1}{k} \sum_{i = 0}^{k - 1} |(\delta(x_i, 0)-\delta(y_i, 0))| \\
			& \leq \frac{\# i \in [0, k - 1] : x_i \neq y_i}{k} \\
			& \stackrel{k \to \infty}{\to} 0 .
		\end{align*}
		Therefore, we can let $S$ be the equivalence class of $x$ under $\sim$. To see that this $S$ is uncountable, let $E \subseteq \mathbb{N}$ be an infinite subset of density $0$. Then $I$ has density $0$. For each $F \subseteq E$, let $x^F \in X$ be a sequence such that $x_j^F = x_j$ for $j \not \in F$ and $x_j^F \neq x_j$ for $j \in F$. Since $E$ has uncountably many subsets, and $x \sim x^F$ for all $F \subseteq E$, we have shown that the equivalence class of $x$ by $\sim$ is uncountable. So, assuming that $x \in X$ such that $\left( \alpha_{C_k(x)} \left( \frac{1}{k} \sum_{i = 0}^{k - 1} f \right) \right)_{k = 1}^\infty$ is not Cauchy, then we can let $S = \{ y \in X : x \sim y \}$.
		
		Our next order of business is to construct some such $x$. The identity $(\dagger)$ also helps us construct an $x \in X$ for which $\left( \alpha_{C_k(x)} \left( \frac{1}{k} \sum_{i = 0}^{k - 1} T^i f \right) \right)_{k = 1}^\infty$ is not Cauchy. Construct $x = (x_j)_{j \in \mathbb{Z}} \in X$ as follows. For brevity, let $c_n = \sum_{p = 1}^n 2^p$. Set
		$$x_j = \begin{cases}
			0	& j < 0 \\
			1	& j = 0 \\
			0	& 0 < j \leq 2 \\
			1	& 2 < j \leq 6 \\
			0	& 6 < j \leq 14 \\
			1	& 14 < j \leq 30 \\
			\vdots \\
			0	& c_{2n} < j \leq c_{2n + 1} \\
			1	& c_{2n + 1} < j \leq c_{2n + 2} \\
			0	& c_{2n + 2} < j \leq c_{2n + 3} \\
			\vdots
		\end{cases}$$
		In plain language, this sequence begins with $0$ for $j < 0$, a $1$ at $j = 0$, then $2^1$ terms of $0$, then $2^2$ terms of $1$, then $2^3$ terms of $0$, then $2^4$ terms of $1$, and so on. We claim that $\liminf_{k \to \infty} \alpha_{C_k(x)} \left( \frac{1}{k} \sum_{i = 0}^{k - 1} T^i f \right) \neq \limsup \alpha_{C_k(x)} \left( \frac{1}{k} \sum_{i = 0}^{k - 1} T^i f \right)$. Sampling along the subsequence $k_n = c_{2n} + 1$, we get
		\begin{align*}
			\alpha_{C_{c_{2n} + 1}(x)} \left( \frac{1}{c_{2n} + 1} \sum_{i = 0}^{c_{2n}} T^i f \right)	& = \frac{2 + 8 + 32 + \cdots + 2^{2n - 1}}{1 + 2 + 4 + 6 + \cdots + 2^{2n}}	& = \frac{\frac{1}{2}\sum_{p = 1}^n 4^p}{1 + \sum_{q = 1}^{2n} 2^q} \\
			& = \frac{1}{3} \cdot \frac{4^n - 1}{4^n - \frac{1}{2}} & \stackrel{n \to \infty}{\to} \frac{1}{3} ,
		\end{align*}
		where the limit is taken using L'Hospital's Rule. On the other hand, looking at the subsequence $k_n = c_{2n - 1} + 1$, we get
		\begin{align*}
			\alpha_{C_{c_{2n - 1} + 1}(x)} \left( \frac{1}{c_{2n - 1} + 1} \sum_{i = 0}^{c_{2n - 1}} T^i f \right)	& = \frac{2 + 8 + 32 + \cdots + 2^{2n - 1}}{1 + 2 + 4 + 6 + \cdots + 2^{2n - 1}}	& = \frac{\frac{1}{2}\sum_{p = 1}^n 4^p}{1 + \sum_{q = 1}^{2n - 1} 2^q} \\
			& = \frac{1}{3} \cdot \frac{4^{n} - 1}{\frac{1}{2} 4^{n} - \frac{1}{2}}	& \stackrel{n \to \infty}{\to} \frac{2}{3} .
		\end{align*}
		Thus we can say
		$$
		\liminf_{k \to \infty} \alpha_{C_k(x)} \left( \frac{1}{k} \sum_{i = 0}^{k - 1} T^i f \right) \leq \frac{1}{3} < \frac{2}{3} \leq \limsup_{k \to \infty} \alpha_{C_k(x)} \left( \frac{1}{k} \sum_{i = 0}^{k - 1} T^i f \right) .
		$$
		Therefore the sequence $\left( \alpha_{C_k(x)} \left( \frac{1}{k} \sum_{i = 0}^{k - 1} T^i f \right) \right)_{k = 1}^\infty$ is divergent, and thus not Cauchy.
	\end{proof}

	\begin{Rmk}
	Theorem \ref{Pathological Differentiation} is not encompassed by Theorem \ref{Pathological connected differentiation}, since a subshift is a priori totally disconnected.
	\end{Rmk}
	
	This calculation adequately sets up the following result.
	
	\begin{Thm}\label{Normal shift points}
		Let $(X, \mathcal{B}, \mu, T)$ be an ergodic subshift, with $X \subseteq \mathcal{D}^\mathbb{Z}$, and let $x \in X$. Then the following statements about $x \in X$ are equivalent.
		\begin{enumerate}
			\item For all $f \in C(X)$, the limit $$\lim_{k \to \infty} \alpha_{C_k(x)}\left( \frac{1}{k} \sum_{i = 0}^{k - 1} T^i f \right)$$
			exists and is equal to $\int f \mathrm{d} \mu$.
			\item For all words $(a_0, a_1, \ldots, a_{\ell -  1}) \in \bigcup_{\ell = 1}^\infty \mathcal{D}^\ell$, the limit $$\lim_{k \to \infty} \alpha_{C_k(x)}\left( \frac{1}{k} \sum_{i = 0}^{k - 1} T^i \chi_{[a_0, a_1, \ldots, a_{\ell -  1}]} \right)$$
			exists and is equal to $\mu([a_0, a_1, \ldots, a_{\ell -  1}])$.
			\item For all words $(a_0, a_1, \ldots, a_{\ell -  1}) \in \bigcup_{\ell = 1}^\infty \mathcal{D}^\ell$, the limit
			$$\lim_{k \to \infty} \frac{\# \{ i \in [0, k - \ell] : x_i = a_0, x_{i + 1} = a_1, \ldots, x_{i + \ell - 1} = a_{\ell - 1} \}}{k}$$
			exists and is equal to $\mu([a_0, a_1, \ldots, a_{\ell - 1}])$.
			\item For all words $(a_0, a_1, \ldots, a_{\ell -  1}) \in \bigcup_{\ell = 1}^\infty \mathcal{D}^\ell$, the limit
			$$\lim_{k \to \infty} \frac{\# \{ i \in [0, k - 1] : x_i = a_0, x_{i + 1} = a_1, \ldots, x_{i + \ell - 1} = a_{\ell - 1} \}}{k}$$
			exists and is equal to $\mu([a_0, a_1, \ldots, a_{\ell - 1}])$. 
		\end{enumerate}
	\end{Thm}
	
	\begin{proof}
		Lemma \ref{Cylinders dense} tells us that (1)$\iff$(2). That (3)$\iff$(4) comes from the observation that the absolute difference between the two sequences is at most $\frac{\ell - 1}{k}$. To establish (2)$\iff$(3), we compute $\alpha_{C_k(x)} \left( T^i \chi_{[a_0, a_1, \ldots, a_{\ell - 1}]} \right)$ for $i \in [0, k - \ell]$ as follows.
		\begin{align*}
			\alpha_{C_k(x)} \left( T^i \chi_{[a_0, a_1, \ldots, a_{\ell - 1}]} \right)	& = \frac{1}{\mu([x_0, x_1, \ldots, x_{k - 1}])} \int_{[x_0, x_1, \ldots, x_{k - 1}]} \chi_{T^{-i} [a_0, a_1, \ldots, a_{\ell - 1}]} \mathrm{d} \mu \\
			& = \begin{cases}
				1	& a_i = x_0, a_{i + 1} = x_1, \ldots, a_{i + \ell - 1} = x_{\ell - 1} , \\
				0	& \textrm{otherwise}
			\end{cases} .
		\end{align*}
		Therefore
		\begin{align*}
			& \frac{1}{k} \sum_{i = 0}^{k - \ell} \alpha_{C_k(x)} \left( T^i \chi_{[a_0, a_1, \ldots, a_{\ell - 1}]} \right) \\
			& =  \frac{\# \{ i \in [0, k - \ell] : x_i = a_0, x_{i + 1} = a_1, \ldots, x_{i + \ell - 1} = a_{\ell - 1} \}}{k} .
		\end{align*}
		Finally, we observe that
		\begin{align*}
			\left| \alpha_{C_k(x)} \left( \frac{1}{k} \sum_{i = 0}^{k - 1} T^i f \right) - \alpha_{C_k(x)} \left( \frac{1}{k} \sum_{i = 0}^{k - \ell} T^i f \right) \right|	& = \left| \alpha_{C_k(x)} \left( \frac{1}{k} \sum_{i = k - \ell + 1}^{k - 1} T^i f \right) \right|	\\
			& \leq \frac{\ell - 1}{k} \| f \|_\infty .
		\end{align*}
		Therefore the end behaviors of $\left( \alpha_{C_k(x)}\left( \frac{1}{k} \sum_{i = 0}^{k - 1} T^i \chi_{[a_0, a_1, \ldots, a_{\ell -  1}]} \right) \right)_{k = 1}^\infty$ and \linebreak $\left( \alpha_{C_k(x)}\left( \frac{1}{k} \sum_{i = 0}^{k - \ell} T^i \chi_{[a_0, a_1, \ldots, a_{\ell -  1}]} \right) \right)_{k = 1}^\infty$ are identical, i.e. one converges iff the other converges, and if they converge, then they converge to the same value. But then, as has already been established, we know that  
		\begin{align*}
		\left( \alpha_{C_k(x)}\left( \frac{1}{k} \sum_{i = 0}^{k - \ell} T^i \chi_{[a_0, a_1, \ldots, a_{\ell -  1}]} \right) \right)_{k = 1}^\infty \\
		= \left( \frac{\# \{ i \in [0, k - 1] : x_i = a_0, x_{i + 1} = a_1, \ldots, x_{i + \ell - 1} = a_{\ell - 1} \}}{k} \right)_{k = 1}^\infty ,
		\end{align*}
		demonstrating that (2)$\iff$(3).
	\end{proof}
	
	Theorem \ref{Normal shift points} gives us an alternate proof of Theorem \ref{Metric result for subshifts}. Applying the Birkhoff Ergodic Theorem to the functions $\chi_{[\mathbf{a}]}$ tells us that almost all $x \in X$ satisfy $\frac{1}{k} \sum_{i = 0}^{k - 1} T^i \chi_{[\mathbf{a}]}(x) \stackrel{k \to \infty}{\to } \mu([\mathbf{a}])$ for all strings $\mathbf{a} \in \bigcup_{\ell = 1}^\infty \mathcal{D}^\ell$. But this is exactly condition (4) from Theorem \ref{Normal shift points}. Moreover, this result gives us a more concrete characterization of the "set of full measure" that Theorem \ref{Metric result for subshifts} alludes to.
	
	Before concluding, we demonstrate that Proposition \ref{Random Lipschitz Cylinders} does not hinge on the cylinder structure of $X$.
	
	\begin{Thm}
		Let $(X, \rho)$ be a compact metric space, and let $T : X \to X$ be an $L$-Lipschitz homeomorphism on $X$ with respect to $\rho$, where $L > 1$. Suppose $\mu$ is a regular Borel probability measure on $X$ such that $T$ is ergodic with respect to $\mu$. Let $(r_k)_{k = 1}^\infty$ be a sequence of positive numbers $r_k > 0$ such that there exists a constant $\gamma \in \mathbb{R}$ such that $r_k \leq \gamma \cdot L^{-k}$ for all $k \in \mathbb{N}$. Fix $f \in C(X)$. Let $B_k(x) = \{ y \in X : \rho(x, y) < r_k \}$. Then the set of $x \in X$ such that
		$$\frac{1}{k} \sum_{i = 0}^{k - 1} \alpha_{B_k(x)} \left( T^i f \right) \stackrel{k \to \infty}{\to } \int f \mathrm{d} \mu$$
		for all $f \in C(X)$ is of full measure.
	\end{Thm}

	\begin{proof}
	Since $C(X)$ is separable, it will suffice to show that given some fixed $f \in C(X)$, we have
	$$\frac{1}{k} \sum_{i = 0}^{k - 1} \alpha_{B_k(x)} \left( T^i f \right) \stackrel{k \to \infty}{\to } \int f \mathrm{d} \mu$$
	for almost all $x \in X$. Our method of proof will closely resemble our proof of Proposition \ref{Random Lipschitz Cylinders}.
	
	Our goal is to show that for every $\epsilon > 0$ exists some $K \in \mathbb{N}$ such that if $k \geq K(\epsilon)$, we have
		\begin{align*}
			\left| \int f \mathrm{d} \mu - \frac{1}{k} \sum_{i = 0}^{k - 1} \alpha_{B_k(x)} \left( T^i f \right) \right|	\\
			\leq \left| \int f \mathrm{d} \mu - \frac{1}{k} \sum_{i = 0}^{k - 1} \left( T^i f \right)(x) \right| + \left| \frac{1}{k} \sum_{i = 0}^{k - 1} \left( \left( T^i f \right)(x) - \alpha_{B_k(x)} \left( T^i f \right) \right) \right| \\
			\leq \left| \int f \mathrm{d} \mu - \frac{1}{k} \sum_{i = 0}^{k - 1} \left( T^i f \right)(x) \right| + \frac{1}{k} \sum_{i = 0}^{k - 1} \left| \left( T^i f \right)(x) - \alpha_{B_k(x)} \left( T^i f \right) \right| \\
			\leq \epsilon .
		\end{align*}
		We will accomplish this by bounding the terms
		$$\left| \int f \mathrm{d} \mu - \frac{1}{k} \sum_{i = 0}^{k - 1} \left( T^i f \right)(x) \right| , \; \frac{1}{k} \sum_{i = 0}^{k - 1} \left| \left( T^i f \right)(x) - \alpha_{B_k(x)} \left( T^i f \right) \right| $$
		by $\epsilon$.
		
		We will start with bounding the latter term. We claim that if $x \in X$ such that $\mu(B_k(x)) > 0, \operatorname{diam}(B_k(x)) \leq \gamma_x \cdot L^{-k}$ for all $k \in \mathbb{N}$, then for every $\epsilon > 0$, there exists $K_1 \in \mathbb{N}$ such that
		$$k \geq K_1 \Rightarrow \frac{1}{k} \sum_{i = 0}^{k - 1} \left| \left( T^i f \right) (x) - \alpha_{B_k(x)} \left( T^i f \right) \right| < \frac{\epsilon}{2} .$$
		To prove this, choose $\delta > 0$ such that $p(y, z) < \delta \Rightarrow |f(y) - f(z)| < \frac{\epsilon}{4}$. Let $\kappa \in \mathbb{N}$ such that $\gamma_x \cdot L^{-\kappa} < \delta$. Then if $k > \kappa$, then
		\begin{align*}
			\frac{1}{k} \sum_{i = 0}^{k - 1} \left| \left( T^i f \right)(x) - \alpha_{B_k(x)} \left( T^i f \right) \right|	& = \frac{1}{k} \left[ \sum_{i = 0}^{k - \kappa} \left| \left( T^i f \right)(x) - \alpha_{B_k(x)} \left( T^i f \right) \right| \right] \\
			& + \frac{1}{k} \left[ \sum_{k - \kappa + 1}^{k - 1} \left| \left( T^i f \right)(x) - \alpha_{B_k(x)} \left( T^i f \right) \right| \right] .
		\end{align*}
		We will estimate these two terms separately, bounding each by $\frac{\epsilon}{4}$. Beginning with the former, we observe that if $x, y \in B_k(x)$, then
		$$p \left( T^{i} x, T^i y \right) \leq L^{i} p(x, y) \leq L^i \cdot \gamma_x \cdot L^{-k} = \gamma_{x} \cdot L^{i - k} .$$
		In particular, this means that if $i - k \leq - \kappa$, then $\left| \left( T^i f \right)(x) - f(z) \right| < \frac{\epsilon}{4}$ for all $z = T^i y \in T^i C_{k}(x)$, so
		\begin{align*}
			& \frac{1}{k} \left[ \sum_{i = 0}^{k - \kappa} \left| \left( T^i f \right)(x) - \alpha_{B_k(x)} \left( T^i f \right) \right| \right] \\
			& = \frac{1}{k} \left[ \sum_{i = 0}^{k - \kappa} \left| \frac{1}{\mu(B_k(x))} \int_{B_k(x)} \left( \left( T^i f \right)(x) \right) - T^i f \mathrm{d} \mu \right| \right] \\
			& = \frac{1}{k} \left[ \sum_{i = 0}^{k - \kappa} \frac{1}{\mu\left(T^i B_k(x)\right)} \int_{T^i B_k(x)} \left| \left( T^i f \right) (x) - f \right| \mathrm{d} \mu \right] \\
			& \leq \frac{1}{k} \left[ \sum_{i = 0}^{k - \kappa} \frac{1}{\mu\left(T^i B_k(x)\right)} \int_{T^i B_k(x)} \frac{\epsilon}{4} \mathrm{d} \mu \right] \\
			& = \frac{k - \kappa + 1}{k} \frac{\epsilon}{4} \\
			& \leq \frac{\epsilon}{4} .
		\end{align*}
		On the other hand, we can estimate
		$$\frac{1}{k} \left[ \sum_{k - \kappa + 1}^{k - 1} \left| \left( T^i f \right)(x) - \alpha_{B_k(x)} \left( T^i f \right) \right| \right]  \leq \frac{2 \kappa}{k} \| f \| .$$
		Choose $K_1 > \kappa$ such that $\frac{2 \kappa \left\| f \right\|_\infty}{K_1} < \frac{\epsilon}{4}$. Then if $k \geq K_1$, we have
		\begin{align*}
			\frac{1}{k} \sum_{i = 0}^{k - 1} \left| \left( T^i f \right)(x) - \alpha_{B_k(x)} \left( T^i f \right) \right|	& = \frac{1}{k} \left[ \sum_{i = 0}^{k - \kappa} \left| \left( T^i f \right)(x) - \alpha_{B_k(x)} \left( T^i f \right) \right| \right] \\
			& + \frac{1}{k} \left[ \sum_{k - \kappa + 1}^{k - 1} \left| \left( T^i f \right)(x) - \alpha_{B_k(x)} \left( T^i f \right) \right| \right] \\
			& \leq \frac{\epsilon}{4} + \frac{\epsilon}{4} \\
			& = \frac{\epsilon}{2} .
		\end{align*}
		
		Now suppose further that $x \in X$ is such that $\frac{1}{k} \sum_{i = 0}^{k - 1} \left( T^i f \right)(x) \stackrel{k \to \infty}{\to} \int f \mathrm{d} \mu$. Choose $K_2 \in \mathbb{N}$ such that $k \geq K_2 \Rightarrow \left| \int f \mathrm{d} \mu - \frac{1}{k} \sum_{i = 0}^{k - 1} \left( T^i f \right)(x) \right| < \frac{\epsilon}{2}$. Then if $k \geq \max \{ K_1, K_2 \}$, then we have
		\begin{align*}
			\frac{1}{k} \sum_{i = 0}^{k - 1} \left| \left( T^i f \right)(x) - \alpha_{B_k(x)} \left( T^i f \right) \right|	& = \frac{1}{k} \left[ \sum_{i = 0}^{k - \kappa} \left| \left( T^i f \right)(x) - \alpha_{B_k(x)} \left( T^i f \right) \right| \right] \\
			& + \frac{1}{k} \left[ \sum_{k - \kappa + 1}^{k - 1} \left| \left( T^i f \right)(x) - \alpha_{B_k(x)} \left( T^i f \right) \right| \right] \\
			& \leq \frac{\epsilon}{2} + \frac{\epsilon}{2} \\
			& = \epsilon .
		\end{align*}
	\end{proof}
	
		Before looking at a more general family of differentiation problems, we want to take a moment to observe that if $(X, \mathcal{B}, T, \mu)$ is an ergodic system, then if the (measure-theoretic) entropy $h(T, \mu)$ of the system is positive, then we automatically have that $\mu(C_k(x)) \stackrel{k \to \infty}{\to} 0$: by the Shannon-McMillan-Breiman Theorem \cite[Theorem 6.2.1]{D&K}, it follows that for $\mu$-almost every $x \in X$ there exists $K = K_x \in \mathbb{N}$ such that
		$$k \geq K \Rightarrow - \frac{1}{k} \log \mu (C_k(x)) \geq \frac{h(T, \mu)}{2} .$$
		Then if $k \geq K$, we have
		\begin{align*}
			- \frac{1}{k} \log \mu(C_k(x))	& \geq \frac{h(T, \mu)}{2} \\
			\Rightarrow \log \mu(C_k(x))	& \leq - \frac{h(T, \mu)}{2} k	& < 0 \\
			\Rightarrow \mu(C_k(x))	& \leq \left( e^{- \frac{h(T, \mu)}{2}} \right)^k	& \stackrel{k \to \infty}{\to} 0 .
		\end{align*}
		On the other hand, whether $\mu(B_k(x)) \stackrel{k \to \infty}{\to} 0$ depends on where $(X, \mathcal{B}, \mu)$ contains atoms. If $\mu(\{x\}) = 0$ for all $x \in X$, then $\mu(B_k(x)) \stackrel{k \to \infty}{\to} 0$.

	\section{Random cylinders in a Bernoulli shift - a probabilistic approach}\label{Bernoulli shifts and probability}
	
	In this section, we consider problems similar to those addressed in Sections \ref{Non-expansive} and \ref{Lipschitz}, where we take some $(X, \mathcal{B}, \mu, T)$ with specified properties (in this case, we assume the system is Bernoulli), and seek to establish conditions under which for a randomly chosen sequence $(F_k)_{k = 1}^\infty$ of sets of positive measure, the sequence $\left( \frac{1}{\mu(F_k)} \int_{F_k} \left( \frac{1}{k} \sum_{i = 0}^{k - 1} T^i f \right) \mathrm{d} \mu \right)_{k = 1}^\infty$ converges almost surely to $\int f \mathrm{d} \mu$ for all $f \in C(X)$.
	
	We provide now an alternate proof of a special case of Theorem \ref{Metric result for subshifts}. Though the result proved is lesser in scope, we include it for the reason that the proof provided here has a decidedly more probabilistic flavor than the proof provided of Theorem \ref{Metric result for subshifts} in Section \ref{Lipschitz}. This method of proof also proves slightly more versatile, as it allows us to consider randomly chosen sequences of cylinders which are not necessarily nested.
	
	In this section, $X = \mathcal{D}^\mathbb{Z}$ is a Bernoulli shift on a finite alphabet $\mathcal{D}$ with probability vector $\mathbf{p} = (p(d))_{d \in \mathcal{D}}$, and $\mu$ is the Borel probability measure on $X$ induced by $\mathbf{p}$. We begin by proving a lemma to which we assign a whimsical title.
	
	\begin{Lem}[The Even Stronger Law of Large Numbers]
		Let $(Y, \mathcal{A}, \nu)$ be a probability space, and let $(k_n)_{n = 1}^\infty$ be a sequence in $\mathbb{N}$ such that $\sum_{n = 1}^\infty k_n^{-2} < \infty$. Let $(\zeta_{i, n})_{0 \leq i \leq k_n - 1, n \in \mathbb{N}}$ be a family of $L^\infty$ real random variables satisfying the following conditions.
		\begin{enumerate}
			\item There exists $C \in [1, \infty)$ such that $\| \zeta_{i, n} \|_\infty \leq C$ for all $0 \leq i \leq k_n - 1, n \in \mathbb{N}$.
			\item $\int \zeta_{i, n} \mathrm{d} \nu = m$ for all $0 \leq i \leq k_n - 1 , n \in \mathbb{N}$, where $m$ is a constant.
			\item For each $n \in \mathbb{N}$, the subfamily $\{ \zeta_{i, n} \}_{i = 0}^{k_n - 1}$ is mutually independent.
		\end{enumerate}
		Then
		$$\frac{1}{k_n} \sum_{i = 0}^{k_n - 1} \zeta_{i, n} \stackrel{n \to \infty}{\to} m$$
		almost surely.
	\end{Lem}
	
	\begin{proof}
		For the sake of brevity, abbreviate
		$$S_n = \sum_{i = 0}^{k_n - 1} \zeta_{i, n},$$
		and assume without loss of generality that $m = 0$ (else, we can just consider $\hat{\zeta}_{i, n} = \zeta_{i, n} - m$). Given $\epsilon > 0$, set
		$$E_{n, \epsilon} = \{ y \in Y : |S_n(y)|/k_n \geq \epsilon \} = \{ y \in Y : |S_n(y)| \geq k_n \epsilon \} .$$
		Then Chebyshev's inequality tells us that
		$$\mu(E_{n, \epsilon}) \leq \frac{1}{(k_n \epsilon)^4} \int S_n^4 \mathrm{d} \nu .$$ Then
		$$\int S_n^4 \mathrm{d} \nu = \sum_{r, s, t, u = 0}^{k_n - 1} \int \zeta_{r, n} \zeta_{s, n} \zeta_{t, n} \zeta_{u, n} \mathrm{d} \nu .$$
		This sum consists of terms of the forms
		\begin{enumerate}
			\item $\int \zeta_{r, n}^4 \mathrm{d} \nu$
			\item $\int \zeta_{r, n}^2 \zeta_{s, n}^2 \mathrm{d} \nu$
			\item $\int \zeta_{r, n}^3 \zeta_{s, n} \mathrm{d} \nu$
			\item $\int \zeta_{r, n}^2 \zeta_{s, n} \zeta_{t, n} \mathrm{d} \nu$
			\item $\int \zeta_{r, n} \zeta_{s, n} \zeta_{t, n} \zeta_{u, n} \mathrm{d} \nu$
		\end{enumerate}
		where $r , s , t , u$ are distinct. We assert that the terms of the third, fourth, and fifth forms all vanish by virtue of independence. This leaves $k_n$ terms of the first form and $3k_n (k_n - 1)$ terms of the second form. Thus there are $k_n + 3 k_n (k_n - 1)$ terms of absolute value $\leq C^4$. Thus
		\begin{align*}
			\int S_n^4 \mathrm{d} \nu & \leq \left( 3k_n^2 - 2 k_n \right)^2 C^4 \\
			& \leq 3 k_n^2 C^4 \\
			\Rightarrow \mu(E_{n, \epsilon}) & \leq \frac{3 k_n^2 C^4}{(k_n \epsilon)^4} \\
			\Rightarrow \sum_{n = 1}^\infty \mu(E_{n, \epsilon})	& \leq \frac{3 C^4}{\epsilon^4}\sum_{n = 1}^\infty k_n^{-2} \\
			& < \infty .
		\end{align*}
		By the Borell-Cantelli Lemma, it follows that $\mu \left( \bigcap_{N = 1}^\infty \bigcup_{n = N}^\infty E_{n, \epsilon} \right) = 0$. But $$\bigcap_{N = 1}^\infty \bigcup_{n = N}^\infty E_{n, \epsilon} = \left\{ y \in Y : \limsup_{n \to \infty} \left| \frac{S_n(y)}{k_n} \right| \geq \epsilon \right\} ,$$
		so we can conclude that
		$$\mu \left( \left\{ y \in Y : \limsup_{n \to \infty} \left| \frac{S_n(y)}{k_n} \right| > 0 \right\} \right) = \mu \left( \bigcup_{K = 1}^\infty \left( \bigcap_{N = 1}^\infty \bigcup_{n = N}^\infty E_{n, \frac{1}{K}} \right) \right) = 0 .$$
		Thus $\frac{S_n}{k_n} \to m$ almost surely.
	\end{proof}
	
	Now we apply this to estimating $$\frac{1}{k} \sum_{i = 0}^{k - 1} \frac{1}{\mu(C_k(x))} \int_{C_k(x)} T^i \chi_{[a_0, a_1, \ldots, a_{\ell - 1}]} .$$ Fix a word $\mathbf{a} = (a_0, a_1, \ldots, a_{\ell - 1}) \in \mathcal{D}^\ell$. We are going to consider a sequence of families of discrete random variables in $X$ given by
	\begin{align*}
		\xi_{i, k}^{\mathbf{a}} (x)	& = \frac{1}{\mu(C_k(x))} \int_{C_k(x)} T^i \chi_{[\mathbf{a}]} \mathrm{d} \mu \\
		& = \frac{\mu \left( C_k(x) \cap T^{-i} [\mathbf{a}] \right)}{\mu(C_k(x))} .	& (0 \leq i \leq k - 1)
	\end{align*}
	Each random variable is bounded in $L^\infty(X, \mu)$ by $1$. We claim that they also have a shared mean $\int \xi_{i, k}^{\mathbf{a}} \mathrm{d} \mu = \mu([\mathbf{a}])$.
	\begin{align*}
		& \int \xi_{i, k}^{\mathbf{a}} \mathrm{d} \mu	\\
		& = \sum_{\vec{d} \in \mathcal{D}^k} \left( \prod_{h = 0}^{k - 1} p(d_h) \right) \alpha_{[d_0, d_1, \ldots, d_{k - 1}]} \left( T^i \chi_{[a_0, a_1, \ldots, a_{\ell - 1}]} \right) \\
		& = \sum_{\vec{d} \in \mathcal{D}^k} \left( \prod_{h = 0}^{k - 1} p(d_h) \right) \frac{1}{\prod_{h = 0}^{k - 1} p(d_h)} \int_{[d_0, d_1, \ldots, d_{k - 1}]} T^i \chi_{[a_0, a_1, \ldots, a_{\ell - 1}]} \mathrm{d} \mu \\
		& = \sum_{\vec{d} \in \mathcal{D}^k} \int_{[d_0, d_1, \ldots , d_{k - 1}]} T^i \chi_{[a_0, a_1, \ldots, a_{\ell - 1}]} \mathrm{d} \mu \\
		& = \sum_{\vec{d} \in \mathcal{D}^k} \int_{[d_0, d_1, \ldots , d_{k - 1}]} \chi_{T^{-i} [a_0, a_1, \ldots, a_{\ell - 1}]} \mathrm{d} \mu \\
		& = \sum_{\vec{d} \in \mathcal{D}^k} \mu \left( [d_0, d_1, \ldots, d_{k - 1}] \cap T^{-i} [a_0, a_1, \ldots, a_{\ell - 1}] \right) \\
		& = \sum_{\vec{d} \in \mathcal{D}^k} \mu \left( [d_0, d_1, \ldots, d_{k - 1}] \cap \bigcup_{c_0, c_1, \ldots, c_{i - 1}} [c_0, c_1, \ldots, c_{i - 1}, a_0, a_1, \ldots, a_{\ell - 1}] \right) \\
		& = \sum_{\vec{d} \in \mathcal{D}^k} \mu \left( [d_0, d_1, \ldots, d_{k - 1}] \cap [d_0, d_1, \ldots, d_{i - 1}, a_0, a_1, \ldots, a_{\ell - 1}] \right)
	\end{align*}
	To compute this value, we look at two cases: where $i + \ell \leq k$, and where $i + \ell \geq k$.
	
	If $i + \ell \leq k$, then
	\begin{align*}
		[d_0, d_1, \ldots, d_{k - 1}] \cap [d_0, d_1, \ldots, d_{i - 1}, a_0, a_1, \ldots, a_{\ell - 1}] \\
		= \begin{cases}
			[d_0, d_1, \ldots, d_{k - 1}]	& d_i = a_0, d_{i + 1} = a_1, \ldots, d_{i + \ell - 1} = a_{\ell - 1} \\
			\emptyset	& \textrm{otherwise}
		\end{cases}
	\end{align*}
	This means that $d_0, d_1, \ldots, d_{i - 1}$, as well as $d_{i + \ell}, \ldots, d_{k - 1}$ are "free". Thus
	\begin{align*}
		\sum_{ \vec{d} \in \mathcal{D}^k} \mu \left( [d_0, d_1, \ldots, d_{k - 1}] \cap [d_0, d_1, \ldots, d_{i - 1}, a_0, a_1, \ldots, a_{\ell - 1}] \right) \\
		= \sum_{ \vec{d} \in \mathcal{D}^l} \mu([d_0, d_1, \ldots, d_{i - 1}, a_0, a_1, \ldots, a_{\ell - 1}, d_{i + \ell}, \ldots, d_{k - 1}]) \\
		= \sum_{ \vec{d} \in \mathcal{D}^k} \left( p(d_0) p(d_1) \cdots p(d_{i - 1}) \right) \left( p(a_0) p(a_1) \cdots p(a_{\ell - 1}) \right) \left( p(d_{i + \ell}) \cdots p(d_{k - 1}) \right) \\
		= \mu([a_0, a_1, \ldots, a_{\ell - 1}]) .
	\end{align*}
	
	On the other hand, if $i + \ell \geq k$, then
	\begin{align*}
		[d_0, d_1, \ldots, d_{k - 1}] \cap [d_0, d_1, \ldots, d_{i - 1}, a_0, a_1, \ldots, a_{\ell - 1}] \\
		= \begin{cases}
			[d_0, d_1, \ldots, d_{i - 1}, a_0, a_1, \ldots, a_{\ell - 1}] & d_i = a_0, \ldots, d_{k - 1} = a_{k - i - 1} \\
			\emptyset	& \textrm{otherwise}
		\end{cases}
	\end{align*}
	leaving $d_0, d_1, \ldots, d_{i - 1}$ "free". Thus
	\begin{align*}
		\sum_{\vec{d} \in \mathcal{D}^k} \mu \left( [d_0, d_1, \ldots, d_{k - 1}] \cap [d_0, d_1, \ldots, d_{i - 1}, a_0, a_1, \ldots, a_{\ell - 1}] \right) \\
		= \sum_{\vec{d} \in \mathcal{D}^k} \mu([d_0, d_1, \ldots, d_{i - 1}, a_0, a_1, \ldots, a_{\ell - 1}]) \\
		= \sum_{\vec{d} \in \mathcal{D}^k} p(d_0) p(d_1) \cdots p(d_{i - 1}) p(a_0) p(a_1) \cdots p(a_{\ell - 1}) \\
		= \mu([a_0, a_1, \ldots, a_{\ell - 1}]) .
	\end{align*}
	
	Thus in either case, we have $\int \xi_{i, k}^{\mathbf{a}} \mathrm{d} \mu = \mu([\mathbf{a}])$.
	
	Now, for fixed $k$, the family $\left\{\xi_{i, k}^{\mathbf{a}} \right\}_{i = 0}^{k - 1}$ is not necessarily independent, but we can break it up into arithmetic subsequences which are. Consider the families $\left\{ \xi_{m \ell + j, k}^\mathbf{a} \right\}_{m = 0}^{\lfloor k / \ell \rfloor - 1}$ for $j \in \{0, 1, \ldots, \ell - 1\}$. Then these subfamilies are independent, so the Even Stronger Law Of Large Numbers tells us that $\frac{1}{\lfloor k / \ell \rfloor} \sum_{m = 0}^{k - 1} \xi_{m \ell + j, k}^\mathbf{a} \to \mu([\mathbf{a}])$ almost surely. Now we calculate
	\begin{align*}
		& \frac{1}{k} \sum_{i = 0}^{k - 1} \frac{1}{\mu(C_k(x))} \int_{C_k(x)} T^i \chi_{[\mathbf{a}]} \mathrm{d} \mu	\\
		& = \frac{1}{k} \sum_{i = 0}^{k - 1} \xi_{i, k}^\mathbf{a}(x)	\\
		& = \frac{\ell \lfloor k / \ell \rfloor}{k} \left[ \frac{1}{\ell \lfloor k / \ell \rfloor} \sum_{i = 0}^{k - 1} \xi_{i, k}^\mathbf{a}(x) \right] \\
		& = \frac{\ell \lfloor k / \ell \rfloor}{k} \left[ \frac{1}{\ell} \sum_{j = 0}^{\ell - 1} \frac{1}{\lfloor k / \ell \rfloor} \sum_{m = 0}^{\lfloor k / \ell \rfloor - 1} \xi_{m \ell + j, k}^\mathbf{a}(x) \right] + \frac{\sum_{i = \ell \lfloor k / \ell \rfloor}^{k - 1} \xi_{i, k}^{\mathbf{a}} (x)}{\ell} \\
		& \stackrel{\textrm{almost surely}}{\to} (1) \left[ \frac{1}{\ell} \sum_{j = 0}^{\ell - 1} \mu([\mathbf{a}]) \right] + 0 \\
		& = \mu([\mathbf{a}]) \\
		& = \int \chi_{[\mathbf{a}]} \mathrm{d} \mu .
	\end{align*}
	Taking a countable intersection over $\mathbf{a} \in \bigcup_{\ell = 1}^\infty \mathcal{D}^\ell$, we can conclude that the set $B$ of all $x \in X$ such that $\frac{1}{k} \sum_{i = 0}^{k - 1} \frac{1}{\mu(C_k(x))} \int_{C_k(x)} T^i \chi_{[\mathbf{a}]} \mathrm{d} \mu \to \int \chi_{[\mathbf{a}]} \mathrm{d} \mu$ for all words $\mathbf{a}$ is of full measure. We can further conclude that if $x \in B$, we have $\frac{1}{\mu(C_k(x))} \int_{C_k(x)} T^i T^n \chi_{[\mathbf{a}]} \mathrm{d} \mu \to \int T^n \chi_{[\mathbf{a}]}$ for all words $\mathbf{a}$ and $n \in \mathbb{Z}$. Since $\operatorname{span} \left\{ T^n \chi_{[\mathbf{a}]} : \mathbf{a} \in \bigcup_{\ell = 1}^\infty \mathcal{D}^\ell, n \in \mathbb{Z} \right\}$ is dense in $C(X)$, we can conclude the following special case of Theorem \ref{Metric result for subshifts}.
	
	\begin{Prop}\label{Bernoulli Prop}
	Let $X = \mathcal{D}^\mathbb{Z}$ be a Bernoulli shift, and let $\mu$ be the associated measure. Endow $X$ with the generator $\mathcal{E} = \{ E_d \}_{d \in \mathcal{D}}$, where $E_d = \{ x \in X : x_0 = d \}$. Then the set of all $x \in X$ such that
	$$\alpha_{C_k(x)} \left( \frac{1}{k} \sum_{i = 0}^{k - 1} T^i f \right) \to \int f \mathrm{d} \mu$$
	for all $f \in C(X)$ is of full measure.
	\end{Prop}
	
	However, this technique lends itself to another result that is not encompassed by Theorem \ref{Metric result for subshifts}. We have looked at spatial-temporal differentiation problems where we are differentiating with respect to the cylinders $C_k(x)$ of a randomly chosen $x \in X$. The next result considers instead the situation where we randomly choose a \emph{sequence} $(x_k)_{k = 1}^\infty$ in $X$ and differentiating with respect to the sequence $(C_k(x_k))_{k = 1}^\infty$. 
	
	\begin{Thm}
		Let $X = \mathcal{D}^\mathbb{Z}$ be a Bernoulli shift, and let $\mu$ be the associated measure. Endow $X$ with the generator $\mathcal{E} = \{ E_d \}_{d \in \mathcal{D}}$, where $E_d = \{ x \in X : x_0 = d \}$. Consider the countably infinite product probability space $\left(X^\infty, \mathcal{B}^\infty, \mu^\infty \right) = \prod_{k \in \mathbb{N}} (X, \mathcal{B}, \mu)$. Then the set of all $(x_k)_{k = 1}^\infty \in X^\infty$ such that
		$$\alpha_{C_k(x_k)} \left( \frac{1}{k} \sum_{i = 0}^{k - 1} T^i f \right) \to \int f \mathrm{d} \mu$$
		for all $f \in C(X)$ is of full $\mu^\infty$-measure.
	\end{Thm}
	
	\begin{proof}
		Our method is very similar to the method used for Proposition \ref{Bernoulli Prop}. Let $\mathbf{x} = (x_k)_{k = 1}^\infty \in X^\infty$ denote a sequence in $X$.
		
		Fix a word $\mathbf{a} = (a_0, a_1, \ldots, a_{\ell - 1}) \in \mathcal{D}^\ell$. We are going to consider a sequence of families of discrete random variables in $X$ given by
		\begin{align*}
			\zeta_{i, k}^{\mathbf{a}} (\mathbf{x})	& = \frac{1}{\mu(C_k(x_k))} \int_{C_k(x_k)} T^i \chi_{[\mathbf{a}]} \mathrm{d} \mu \\
			& = \frac{\mu \left( C_k(x_k) \cap [\mathbf{a}] \right)}{\mu(C_k(x_k))} .	& (0 \leq i \leq k - 1)
		\end{align*}
		Each random variable $\zeta_{i, k}^{\mathbf{a}}$ is bounded in $L^\infty(X, \mu)$ by $1$. By a calculation identical to the one used to prove Proposition \ref{Bernoulli Prop}, we can conclude that $\int \zeta_{i, k}^{\mathbf{a}} \mathrm{d} \mu = \mu([\mathbf{a}])$.
		
		As before, for fixed $k$, the family $\left\{\zeta_{i, k}^{\mathbf{a}} \right\}_{i = 0}^{k - 1}$ is not necessarily independent, but we can break it up into arithmetic subsequences which are. Consider the families $\left\{ \zeta_{m \ell + j, k}^\mathbf{a} \right\}_{m = 0}^{\lfloor k / \ell \rfloor - 1}$ for $j \in \{0, 1, \ldots, \ell - 1\}$. Then these families are independent, and so the Even Stronger Law Of Large Numbers tells us that $\frac{1}{\lfloor k / \ell \rfloor} \sum_{m = 0}^{k - 1} \zeta_{m \ell + j, k}^\mathbf{a} \to \mu([\mathbf{a}])$ almost surely. Now we calculate
		\begin{align*}
			& \frac{1}{k} \sum_{i = 0}^{k - 1} \frac{1}{\mu(C_k(x_k))} \int_{C_k(x_k)} T^i \chi_{[\mathbf{a}]} \mathrm{d} \mu	\\
			& = \frac{1}{k} \sum_{i = 0}^{k - 1} \zeta_{i, k}^\mathbf{a}(\mathbf{x})	\\
			& = \frac{\ell \lfloor k / \ell \rfloor}{k} \left[ \frac{1}{\ell \lfloor k / \ell \rfloor} \sum_{i = 0}^{k - 1} \zeta_{i, k}^\mathbf{a}(\mathbf{x}) \right] \\
			& = \frac{\ell \lfloor k / \ell \rfloor}{k} \left[ \frac{1}{\ell} \sum_{j = 0}^{\ell - 1} \frac{1}{\lfloor k / \ell \rfloor} \sum_{m = 0}^{\lfloor k / \ell \rfloor - 1} \zeta_{m \ell + j, k}^\mathbf{a}(\mathbf{x}) \right] + \frac{\sum_{i = \ell \lfloor k / \ell \rfloor}^{k - 1} \zeta_{i, k}^{\mathbf{a}} (\mathbf{x})}{\ell} \\
			& \stackrel{\textrm{almost surely}}{\to} (1) \left[ \frac{1}{\ell} \sum_{j = 0}^{\ell - 1} \mu([\mathbf{a}]) \right] + 0 \\
			& = \mu([\mathbf{a}]) \\
			& = \int \chi_{[\mathbf{a}]} \mathrm{d} \mu .
		\end{align*}
		Again, taking a countable intersection over $\mathbf{a} \in \bigcup_{\ell = 1}^\infty \mathcal{D}^\ell$, we can conclude that the set $B$ of all $\mathbf{x} \in X^\infty$ such that $\frac{1}{k} \sum_{i = 0}^{k - 1} \frac{1}{\mu(C_k(x_k))} \int_{C_k(x_k)} T^i \chi_{[\mathbf{a}]} \mathrm{d} \mu \to \int \chi_{[\mathbf{a}]} \mathrm{d} \mu$ for all words $\mathbf{a}$ is of full measure. We can further conclude that if $\mathbf{x} \in B$, we have $\frac{1}{\mu(C_k(x_k))} \int_{C_k(x_k)} T^i T^n \chi_{[\mathbf{a}]} \mathrm{d} \mu \to \int T^n \chi_{[\mathbf{a}]}$ for all words $\mathbf{a}$ and $n \in \mathbb{Z}$. Since $\operatorname{span} \left\{ T^n \chi_{[\mathbf{a}]} : \mathbf{a} \in \bigcup_{\ell = 1}^\infty \mathcal{D}^\ell, n \in \mathbb{Z} \right\}$ is dense in $C(X)$, we can conclude that if $\mathbf{x} \in B$, then
		$$\alpha_{C_k(x_k)} \left( \frac{1}{k} \sum_{i = 0}^{k - 1} T^i f \right) \to \int f \mathrm{d} \mu$$
		for all $f \in C(X)$.
	\end{proof}
	\bibliography{Bibliography}
\end{document}